\documentclass[12pt,reqno]{amsart}
\usepackage{amsmath,amssymb,bbm,esint,times,url,algorithm}
\usepackage[noend]{algpseudocode}

\usepackage{tikz}
\usetikzlibrary{calc}
\usepackage[bookmarks=false,unicode,colorlinks,urlcolor=red,citecolor=red,linkcolor=blue]{hyperref}
\usepackage{aliascnt}
\usepackage{booktabs}

\usepackage[margin=1in]{geometry}

\usepackage{subfig}
\newtheorem*{remark}{Remark}
\newtheorem*{remarks}{Remarks}

\newtheorem*{definition}{Definition}

\newtheorem{theorem}{Theorem}[section]

\newaliascnt{lemma}{theorem}
\newtheorem{lemma}[lemma]{Lemma}
\aliascntresetthe{lemma}

\newaliascnt{proposition}{theorem}
\newtheorem{proposition}[proposition]{Proposition}
\aliascntresetthe{proposition}

\newaliascnt{corollary}{theorem}
\newtheorem{corollary}[corollary]{Corollary}
\aliascntresetthe{corollary}

\newaliascnt{conjecture}{theorem}
\newtheorem{conjecture}[conjecture]{Conjecture}
\aliascntresetthe{conjecture}

 \newcommand{\Tau}{\mathcal{T}}

\newcommand{\arxiv}[1]{%
 \href{http://front.math.ucdavis.edu/#1}{ArXiv:#1}}
\newcommand{\mref}[1]{%
\href{http://www.ams.org/mathscinet-getitem?mr=#1}{#1}}

\usepackage{todonotes}

\begin{document}

\title[]{Nearly radial Neumann eigenfunctions on symmetric domains}
\author[]{Nilima Nigam, Bart{\l}omiej Siudeja and Benjamin Young}

\address{Department of Mathematics, Simon Fraser University}
\email{nigam\@@math.sfu.ca}
\address{Department of Mathematics, Univ.\ of Oregon, Eugene,
OR 97403, U.S.A.}
\email{siudeja\@@uoregon.edu}
\address{Department of Mathematics, Univ.\ of Oregon, Eugene,
OR 97403, U.S.A.}
\email{bjy\@@uoregon.edu}


\begin{abstract}
    We study the existence of Neumann eigenfunctions which do not change sign on the boundary of some special domains. We show that eigenfunctions which are strictly positive on the boundary exist on regular polygons with at least 5 sides, while on equilateral triangles and cubes it is not even possible to find an eigenfunction which is nonnegative on the boundary. 

    We use analytic methods combined with symmetry arguments to prove the result for polygons with six or more sides. The case for the regular pentagon is harder. 
We develop a validated numerical method to prove this case, which involves iteratively bounding eigenvalues for a sequence of subdomains of the triangle.  We use a learning algorithm to find and optimize this sequence of subdomains, making it straightforward to check our computations with standard software.
\end{abstract}

\maketitle
\section{Introduction}
We study the existence of an eigenfunction of the Neumann Laplacian which is positive (or nonnegative) on the boundary of  highly symmetric domains. Recently, Hoffmann-Ostenhof \cite{HO13} proved that on rectangles, any Neumann eigenfunction that is positive on the boundary must be constant. In this paper we prove similar results for regular polygons and higher dimensional boxes. 

Schiffer's conjecture (see \cite{Yau}) states that if a Neumann eigenfunction is constant on the boundary of a domain, then either the eigenfunction is constant in the domain, or the domain must be a disk. This conjecture is still open, although many partial results are known (see e.g. \cite{Be80,BY82,De12}). We relax the boundary restriction (positive instead of constant) and ask if the modified conjecture holds for a special class of domains. 

Our problem has a rather interesting physical interpretation in terms of a sloshing liquid in a cup with a uniform, highly symmetric cross-section (see \cite{KKS13} for a relation between sloshing and Neumann eigenvalue problem). It is nearly obvious that one can disturb a fluid in a round cup so that the created wave is radial. In particular the fluid level rises and lowers simultaneously along the whole cup wall.
Hoffmann-Ostenhof \cite{HO13} proved that it is possible to create a wave in a square cup so that there are a few stationary points along the wall, but it is impossible to make all points move in unison. We prove that no such wave can be created in a triangular cup, even if stationary points are allowed. At the same time, it is possible to create a wave with unison movement along the boundary of regular polygons with at least 5 sides.  In summary, we have:

\begin{theorem}\label{thm:equilateral}
  Any Neumann eigenfunction that is \textbf{nonnegative} on the boundary of an equilateral triangle is constant inside.
\end{theorem}

\begin{theorem}[Hoffmann-Ostenhof \cite{HO13}]
  Any Neumann eigenfunction that is \textbf{positive} on the boundary of a rectangle is constant inside.
\end{theorem}

\begin{theorem}\label{thm:regular}
  There exists a Neumann eigenfunction on a regular polygon with $n\ge 5$ sides that is \textbf{positive} on the boundary and not constant.
\end{theorem}

\begin{remark}
Squares are in some sense a critical case for regular polygons. An eigenfunction that is positive on the boundary does not exist, yet
\begin{align*}
  \varphi(x,y)=-\cos \pi x-\cos \pi y 
\end{align*}
is an eigenfunction of the square $[-1,1]^2$. It is positive on the boundary, except at the midpoints of all sides (where it equals 0).
\end{remark}

We also study higher dimensional boxes. Surprisingly, cubes no longer have nonnegative eigenfunctions.
\begin{theorem}\label{thm:cubes}
  Any Neumann eigenfunction that is nonnegative on the boundary of a cube (more generally a box) in dimension $d>2$ is constant inside.
\end{theorem}

Careful Finite Element computations suggest the following conjecture.
\begin{conjecture}
  Eigenfunctions which are nonnegative on the boundary of a tetrahedron and octahedron do not exist. However, eigenfunctions which are positive on the boundary exist on dodecahedron and icosahedron. 
\end{conjecture}

The paper uses a variety of methods to handle the different cases. In particular, we use combinatorial and number theoretic results on cubes (\autoref{sec:cubes}) and equilateral triangles (\autoref{sec:equilateral}). We dissect regular polygons with $n\ge 5$ into congruent right triangles and study their Neumann eigenfunctions. For $n\ge6$ we can use existing results on the shape of the second Neumann eigenfunction to draw the necessary conclusions (\autoref{sec:nge6}).

The proof for the regular pentagon is rather unusual and the most interesting part of the paper. We prove that the nodal line for the second Neumann eigenfunction of a right triangle must connect two longest sides. This seemingly simple fact is extremely hard to prove. Similar results for obtuse triangles have been obtained by Atar and Burdzy \cite{AB02} using very sophisticated probabilistic techniques.

Our proof uses a generalization of the recently-described computable lower bounds for eigenvalue approximations \cite{CG14} in the setting of validated numerics, see \autoref{sec:NonconformingFiniteElements}. We find lower bounds for mixed eigenvalues of 20 complicated polygonal domains and use these in a iterative procedure to restrict the shape of the eigenfunction. The proof itself is human-readable and fully analytic (\autoref{sec:pentagon}), except for the matrix eigenvalue computations on 20 large matrices from \autoref{sec:matrixlowerbound} and \autoref{validatedcomputing}. We employ three different methods (intentional redundancy to ensure correctness, using different software packages) to find lower bounds for the smallest eigenvalues of these matrices: $LDL'$ decomposition with interval arithmetic, $LU$ decomposition using exact rational representations, and Hessenberg decomposition with interval arithmetic and Sturm sequences.

The 20 domains mentioned above were found using a non-validated learning algorithm quickly examining thousands of cases and producing a sequence of domains for the main validated numerical algorithm. Both procedures use a novel approach which iteratively constrains the location of the nodal line for the second Neumann eigenvalue in the original triangular domain, \autoref{sec:algorithm}. Our learning algorithm is not unique in that one could design a different strategy resulting in a different set of the domains, leading to a different validated proof.

\section{Definitions and auxiliary results}
The Neumann eigenvalue problem can be approached classically, by solving the partial differential equation
\begin{align*}
  \Delta u_n&=-\mu_n u_n\text{ in }D,\\
  \partial_\nu u_n&=0\text{ on }\partial D.
\end{align*}
However, it is often more useful to work with the variational weak formulation
\begin{align}
\label{eqn:Rayleigh-Ritz quotient}
  \mu_n=\inf_{
  \mbox{\tiny $\begin{matrix}
   S\subset W^{1,2}(D)\\\dim S=n 
 \end{matrix}$}
  } \sup_{u\in S} \frac{\int_D |\nabla u|^2}{\int_D u^2},
\end{align}
where $H^1(D)$ is the Sobolev space of all functions $u\in L^2(D)$ such that $\nabla u\in (L^2(D))^2$.  The right side of \eqref{eqn:Rayleigh-Ritz quotient} is commonly called the \emph{Rayleigh-Ritz} quotient.  In this context the minimizers of the Rayleigh quotient are the eigenfunctions. For Lipschitz domains (and even more general domains for which appropriate Sobolev embeddings exist) the two approaches lead to the same eigenvalues and the same eigenfunctions (via elliptic regularity considerations). For a broad overview on this topic see \cite{Ba80} and \cite{BB92}.

Note that the variational characterization lacks any obvious boundary conditions. This is a consequence of the Neumann (also called natural) boundary condition being automatically enforced by the Sobolev spaces. In contrast, to enforce Dirichlet boundary condition one seeks minimizers of the Rayleigh-Ritz quotient over a subset of $H^1(D)$ consisting of functions with zero trace on the appropriate part of the boundary of the domain.

In general it is true that
\begin{align*}
  0=\mu_1<\mu_2\le\mu_3\le\dots\le \mu_n\to \infty.
\end{align*}
However in some special cases one can show that $\mu_2$ is simple. In particular,
\begin{lemma}[{\cite[Theorem 1]{S13}}]\label{simple}
For non-equilateral triangles $\mu_2$ is simple.
\end{lemma}

We will also work with a mixed Dirichlet-Neumann eigenvalue problem:
\begin{align*}
  \Delta u_n&=-\lambda_n u_n\text{ in }D,\\
  u_n&=0\text{ on }B\subset \partial D,\\
  \partial_\nu u_n&=0\text{ on }\partial D\setminus B.
\end{align*}
Note that eigenfunctions satisfy Dirichlet boundary conditions on $B$, and the appropriate variational formulation must include the same restriction.
\begin{align*}
  \lambda_n=\inf_{\substack{
     S\subset H^1_B(D)\\\dim S=n }}
  \sup_{u\in S} \frac{\int_D |\nabla u|^2}{\int_D u^2},
\end{align*}
where $H^1_B(D)$ is a subspace of $H^1(D)$ consisting of all functions satisfying $u=0$ on $B$. If $meas(B)>0$, we have
\begin{align*}
  0<\lambda_1<\lambda_2\le\dots\le \lambda_n\to \infty.
\end{align*}

In what follows we need a few geometric results from \cite{S13}.
\begin{lemma}[ {\cite[Lemma 4]{S13}} ]\label{symmetric}
  Suppose $D$ is a domain with a line of symmetry. Then there cannot be two orthogonal antisymmetric eigenfunctions in the span of the eigenspaces of $\mu_2$ and $\mu_3$ (note that $\mu_2$ might equal $\mu_3$).
\end{lemma}

\begin{lemma}[ {Special case of \cite[Lemma 5]{S13}} ]\label{line}
  The nodal line for the second Neumann eigenfunction on a triangle must start on one side and end on another side or vertex connecting the other two sides.
\end{lemma}

Let us introduce the following naming convention for isosceles triangles: 
\begin{definition}
    A triangle is superequilateral (subequilateral) if it is isosceles with aperture angle larger (smaller) than $\pi/3$. 
\end{definition}
Extensive numerical studies suggest that \autoref{line} can be strengthened to:
\begin{conjecture}\label{con:nodal}
  The nodal line for the second Neumann eigenfunction ends in a vertex only for superequilateral triangles. For all other non-equilateral triangles the nodal line connects two longest sides.
\end{conjecture}

We prove this conjecture for some right triangles: namely, the ones whose smallest angle is approximately $\pi/5$. Note that an even stronger conjecture was posed by Atar and Burdzy \cite[Conjecture 3.2]{AB02} for obtuse triangles, where the nodal line would be confined in the right triangle bounded by two long sides and the altitude perpendicular to the longest side.

\newcommand{\regular}[1]{
\pgfmathsetmacro{\n}{#1}
\pgfmathsetmacro{\a}{360/\n}
\coordinate (a) at (0,0);
\foreach \x in {0,\a,...,360} 
{
   \fill[gray,opacity=0.3] ($(a)!0.5!(\x:1)$) -- (a) -- (0,0);
   \draw (a) -- (\x:1) coordinate (a) -- (0,0);
}
}

\newcommand{\regularlabelled}[1]{
\pgfmathsetmacro{\n}{#1}
\pgfmathsetmacro{\a}{360/\n}
\coordinate (a) at (0,0);
\foreach \x in {0,\a,...,360} 
{
   \fill[gray,opacity=0.3] ($(a)!0.5!(\x:1)$) -- (a) -- (0,0);
   \draw (a) -- (\x:1) coordinate (a) -- (0,0);
}
\draw (0:1) node [right=-2pt] {\tiny $D$};
\draw ($(0:1)!0.5!(\a:1)$) node [above right=-4pt] {\tiny $O$};
\draw (\a:1) node [above right=-3pt] {\tiny $B$};
\draw (0,-0.02) node [above right=-1pt] {\tiny $A$};
}

\section{Equilateral triangle: proof of \autoref{thm:equilateral}}\label{sec:equilateral}
The Neumann spectrum for an equilateral triangle can be split into symmetric modes $\varphi_{m,n}$ and antisymmetric modes $\psi_{m,n}$, forming eigenspaces of eigenvalues $\lambda_{m,n}$, $n\ge m\ge 0$. Note that $\lambda_{m,n}$ might be equal for different pairs of integers $m,n$. This means that there exist eigenfunctions of an equilateral triangle that combine many different modes. For more details see McCartin \cite{MC02}. 

Theorem 8.1 from \cite{MC02} states that the symmetric modes never vanish, while the antisymmetric modes degenerate when $m=n$. Furthermore, Theorem 8.2 from~\cite{MC02} gives that both types of modes are rotationally symmetric if and only if $m\equiv n\;(\!\!\!\!\mod 3)$. Finally, any symmetric mode that is not rotationally symmetric can be written as a sum of two rotated (by 120 and 240 degrees) antisymmetric modes. More precisely, denoting the rotations of $\psi_{m,n}$ by $\psi_{m,n,120}$ and $\psi_{m,n,240}$ we have $\varphi_{m,n}=\psi_{m,n,120} + \psi_{m,n,240}$.

Suppose $f$ is an eigenfunction for some $\lambda$ for an equilateral triangle with horizontal side $s_1$. Then $f$ is a linear combination of the symmetric and antisymmetric modes (with respect to the altitude $a_1$ perpendicular to $s_1$).  If $\varphi_{m,n}$ is symmetric but not rotationally symmetric, we rewrite it using two antisymmetric modes. Therefore
\begin{align*}
  f=\sum_{m,k\ge 0} a_{m,m+3k} \varphi_{m,m+3k} + \sum_{m\not\equiv n\!\!\!\!\!\pmod{3}} a_{m,n}(\psi_{m,n,120}+\psi_{m,n,240})+\sum_{m\ne n}b_{m,n}\psi_{m,n}.
\end{align*}

Suppose $f$ is nonnegative on the boundary; let $G$ be the group of isometries of the equilateral triangle. Then $f\circ U$, $U\in G$ is also nonnegative on the boundary. Furthermore, the orbit of $G$ on any antisymmetric mode $\psi$ has size $6$ (all possible rotations and reflections are different) or length $2$ (if $\psi$ is rotationally symmetric). At the same time, the sum of the reflections of $\psi$ along the line of antisymmetry, and in particular at the midpoints of the boundary edges, cancel out. Therefore the function 
\begin{align}\label{eq:generalform}
  F=\sum_{U\in G} f\circ U = 6 \sum_{m,k\ge 0} a_{m,m+3k} \varphi_{m,m+3k}.
\end{align}
is also nonnegative on the boundary. It follows that \autoref{thm:equilateral} needs to be proved only for eigenfunctions of the form $F$. 

Consider the equilateral triangle with vertices $(0,0)$, $(1,0)$ and $(1/2,\sqrt{3}/2)$. We have the following symmetric modes (see e.g. \cite{MC02})
\begin{align}
  \varphi_{m,n}(x,y)&=(-1)^{m+n} \cos  \left(\frac{1}{3}\pi  (2 x-1) (m-n)\right)\cos  \left(\frac2{\sqrt{3}} \pi (m+n)y\right)+\nonumber
  \\&\qquad+(-1)^m\cos\left(\frac{1}{3}\pi  (2 x-1) (m+2n)\right)\cos\left(\frac2{\sqrt{3}} \pi my\right)+\label{eq:symeigf}
  \\&\qquad+(-1)^n\cos\left(\frac{1}{3}\pi  (2 x-1) (2m+n)\right)\cos\left(\frac2{\sqrt{3}} \pi my\right),\nonumber
\end{align}
with corresponding eigenvalues
\begin{align*}
  \lambda_{m,n}=\frac{16\pi^2}{9}(m^2+mn+n^2).
\end{align*}
Note that an eigenvalue $\lambda$ might have a high multiplicity, since many different pairs of integers $(m,n)$ might produce the same value $\lambda$. Therefore more than one pair of integers might belong to a particular eigenvalue $\lambda$.

We find that the rotationally symmetric modes satisfy
\begin{align}\label{eq:trisymm}
    \varphi_{m,m+3k}(x,0)=\cos(2\pi kx)+\cos(2\pi(m+2k)x)+\cos(2\pi(m+k)x),\quad\text{on }[0,1].
\end{align}
Note that we only need to consider the values of $\varphi_{m,m+3k}$ on one side, so we assumed $y=0$. This last quantity integrates to $0$ on $(0,1)$, the side of the triangle, unless $k=0$. This means that a pair $(M,M)$ must belong to $\lambda$ if the eigenfunction associated to $\lambda$ has nonnegative boundary values. This implies that all other pairs $(m,n)$ for the same $\lambda$ satisfy $3M^2=m^2+mn+n^2$.

\subsection{Eigenfunctions positive on the boundary.}
\label{sec:trieigpositive}

Note that if both $m$ and $n$ are even, then $m^2+mn+n^2$ is even, otherwise it is odd. Therefore an even $M$ implies even $m, n$. By induction, any pair $(m,n)$ belonging to eigenvalue $\lambda$ must have both $m$ and $n$ divisible by $2^s$ and at least one not divisible by $2^{s+1}$, whenever $4^s$ divides $M^2$, but $2\cdot 4^s$ does not. Hence there exists $s$ such that for any pair $(m,m+3k)$ that belongs to $\lambda$ we must have $m=2^s m_1$ and $k=2^s k_1$, where at most one of the $k_1$ and $m_1$ is even. Therefore
\begin{align*}
  \varphi_{m,m+3k}(2^{-s-1},0)&=\cos(\pi k_1)+\cos(\pi(m_1+2k_1))+\cos(\pi(m_1+k_1))=
  \\&=(-1)^{k_1}+(-1)^{m_1}+(-1)^{m_1+k_1}=-1,\\
  \varphi_{m,m+3k}(0,0)&=3.
\end{align*}
We have proved that all $\varphi_{m,m+3k}$ that belongs to $\lambda$ must equal $-1$ at the same point on the boundary (and $3$ at vertices). Hence any linear combination of such eigenfunctions with $\sum_{m,k} a_{m,m+3k}\ne 0$ in $F$ must change sign. Also, unconditionally the eigenfunction cannot be strictly positive. 

This argument is very similar to the one used by Hoffmann-Ostenhof \cite{HO13} on squares to prove nonexistence of eigenfunctions positive on the boundary. It is however impossible rule out the existence of eigenfunctions nonnegative on the boundary using this method. We might be able to prove that the linear combination must equal $0$ at many points, but not that it changes sign. 

\subsection{Eigenfunctions nonnegative on the boundary.}
\label{sec:trieignonneg}
Here we develop an improved method based on the fact that Neumann eigenfunction cannot vanish on an open subset of the boundary (it would satisfy Dirichlet condition). Therefore, if we find an open set on which eigenfunction integrates to $0$, it must change sign on that set.

We ony need to work with $\lambda = \frac{16\pi^2}{9} 3M^2$ ($m=M$, $k=0$ is admissible in \eqref{eq:trisymm}, see comment below that formula). It is also possible that other pairs $(m,k)$ with $k>0$ give the same $\lambda$. In that case we have
\begin{align*}
    3M^2 = 3m^2+3mk+k^2.
\end{align*}
It is easy to check that 
\begin{align}\label{eq:mkineq}
M<m+k<m+2k<2M, \qquad 0<k<M.
\end{align}

Consider points 
\begin{align*}
    x_i = \frac{2i+1}{2M},\qquad 0\le i<M,
\end{align*}
and integrals over symmetric intervals around these points
\begin{align*}
    \sum_{i=0}^{M-1} \int_{x_i+a}^{x_i-a} \cos(\alpha x) \,dx = \frac{2}{\alpha} \sin(\alpha a) \sum_{i=0}^{M-1} \cos(\alpha x_i)=
    \frac{2}{\alpha}\sin(\alpha a)\cos(\alpha/2)\sin(\alpha/2)\csc(\alpha/2M)
\end{align*}
by \cite[Section 1.341, Formula 3]{GR00}, as long as $\sin(\alpha/2M)\ne 0$.

Examining formula \eqref{eq:trisymm} we find $\alpha=2\pi k$, $2\pi(m+k)$ and $2\pi(m+2k)$. In each case $\sin(\alpha/2M)\ne 0$ due to \eqref{eq:mkineq}. Furthermore $\sin(\alpha/2)=0$, hence the eigenfunctions $\varphi_{m,m+3k}$ with $k>0$ integrate to $0$ over the union of $(x_i-a,x_i+a)$. We only need to show the same property for for the eigenfunction with $k=0$:
\begin{align}
    \varphi_{M,M}(x,0) = 1+2\cos(2\pi Mx). \label{MMmode}
\end{align}
We have
\begin{align*}
    \sum_{i=0}^{M-1} \int_{x_i+a}^{x_i-a} \varphi_{M,M}(x,0) \,dx &= \sum_{i=0}^{M-1} \left(2a + \frac{2}{\pi M} \sin(2\pi M a)\cos((2i+1)\pi)  \right)
    \\&= 2aM -\frac{2}{\pi} \sin(2\pi M a).
\end{align*}
Let $z_0=2\pi M a_0$ and find a positive solution of $z_0=2\sin(z_0)$. We get $z_0<\pi$ and $a_0<1/2M$. Hence intervals $(x_i-a_0,x_i+a_0)$ fit inside $(0,1)$, the side of the equilateral triangle. At the same time, any linear combination of eigenfunctions from \eqref{eq:trisymm} integrates to $0$ over the union of these intervals. Hence it must change sign, as it cannot satisfy both Dirichelt and Neumann conditon on any interval (be identically $0$ on any interval).

\section{Cubes: proof of \autoref{thm:cubes}.}\label{sec:cubes}

\subsection{Fully symmetric eigenfunctions}

Consider the cube $C=[-1,1]^n$. Suppose it has a Neumann eigenfunction that is positive (nonnegative) on the boundary. We can symmetrize this eigenfunction by applying all isometries of the cube and summing the resulting eigenfunctions (as in the equilateral triangle case). We will get a new eigenfunction that is positive (nonnegative) on the boundary, symmetric with respect to $x_i=0$ for any $i$ and invariant under arbitrary permutation of variables $x_i$. We only need to prove that this fully symmetric eigenfunction cannot be positive (nonnegative) on the boundary.

Any symmetric eigenfunction can be written as a sum of simple eigenfunctions of the form
\begin{align*}
  (-1)^{\sum m_i}\prod_{i=1}^n \cos(m_i\pi x_i).
\end{align*}
The factor $(-1)^{\sum m_i}$ ensures positivity in all vertices ($x_i=\pm1$). Invariance under permutations of variables gives

\begin{align}\label{geneig}
  \varphi_\lambda(x)=\sum_{\stackrel{\sum m_i^2=\lambda}{M=\{m_1\le\dots\le m_n\}}} a_M(-1)^{\sum m_i}\sum_{\sigma_n} 
  \prod_{i=1}^n \cos(m_{\sigma_n(i)}\pi x_i),
\end{align}
where $\sigma_n$ denotes any permutation of $\{1,\dots,n\}$ and $a_M$ are arbitrary coefficients depending on the nondecreasing sequence of nonnegative integers $m_i$. We require that $\sum m_i^2=\lambda$ to ensure all terms belong to the same eigenvalue.  Formula \eqref{geneig} gives the most general form of the eigenfunction that is invariant under the group of the isometries of the cube $C$.
We need to show that this eigenfunction is negative somewhere on the boundary of the cube, regardless of the choice of $\lambda$. Due to symmetry we only need to consider one face. 

Note also, that $\varphi_\lambda$ is also a linear combination of eigenfunctions of the lower dimensional Laplacian on a face. Indeed, fixing $x_1=1$ gives a sum of products of cosines, hence again a symmetric function. However, due the presence of the permutations $\sigma_n$, we drop different $m_i$ in different terms, and hence we get a sum of eigenfunctions for various eigenvalues. Every non-constant Neumann eigenfunction is orthogonal to the constant eigenfunction. Hence $\varphi_\lambda$ integrates to $0$ on each face, unless an eigenfunction which is constant on the face is a part of $\varphi_\lambda$, cf. the discussion below \eqref{eq:trisymm} pertaining to equilateral triangles.

Therefore the sequence $m_1=\dots=m_{n-1}=0$, $m_n=m=\sqrt{\lambda}$ gives one of the terms in $\varphi_\lambda$. Consequently, $\lambda=m^2$ for some integer $m$. Otherwise $\varphi_\lambda$ integrates to $0$ over any face, hence it must change sign on each face.

\subsection{Positive eigenfunctions}
We begin with a special case to illustrate the approach. Suppose $\lambda=m^2$ with odd $m$. Recall that $\varphi_\lambda$ is a sum over all sequences $m_1\le m_2\le \dots \le m_n$ such that 
\begin{align*}
m_1^2+\dots+m_n^2=\lambda=m^2. 
\end{align*}
Hence at least one $m_i$ is odd. Consider a discrete set of points:
\begin{align*}
  X=\{(x_1\ge x_2\ge\dots\ge x_n):x_i\in\{0,1\}\}.
\end{align*}

These points correspond to the center of the cube $(0,\dots,0)$, the center of the face $(1,0,\dots,0)$, the centers of all lower dimensional faces, finally a vertex $(1,\dots,1)$. Let $X_0$ be the set
\begin{align*}
  X_0=\{(1=x_1\ge x_2\ge\dots\ge x_n=0):x_i\in\{0,1\}\}.
\end{align*}
Note that all points in $X_0$ are on one face of the cube. For any $x\in X$ put $k=\sum_{i=1}^n x_i$ (the codimension of the face for which $x$ is a center).  Observe that
\begin{align*}
  \sum_{x\in X_0} \frac{1}{(n-\sum x_i)!}\varphi_\lambda(x)&
  =\sum_{k=1}^{n-1} \sum_{\stackrel{\sum m_i^2=\lambda}{M=\{m_1\le\dots\le m_n\}}} a_M(-1)^{\sum m_i}\frac{1}{(n-k)!}\sum_{\sigma_n} 
  \prod_{i=1}^k \cos(m_{\sigma_n(i)}\pi).
  \\&=\sum_{k=1}^{n-1} \sum_{\stackrel{\sum m_i^2=\lambda}{M=\{m_1\le\dots\le m_n\}}} a_M(-1)^{\sum m_i}\frac{1}{(n-k)!}\sum_{\sigma_n} 
  (-1)^{\sum_{i=1}^k m_{\sigma_n(i)}}.
  \\&=\sum_{\stackrel{\sum m_i^2=\lambda}{M=\{m_1\le\dots\le m_n\}}} a_M(-1)^{\sum m_i}\sum_{k=1}^{n-1} \frac{1}{(n-k)!}\sum_{\sigma_n} 
  (-1)^{\sum_{i=1}^k m_{\sigma_n(i)}}.
\end{align*}
Note that in the innermost sum each term appears exactly $(n-k)!$ times, since we are using only the first $k$ values of each $\sigma_n$. Hence we are adding (exactly once) all products of $(-1)^{m_i}$, except for the full and empty product, so we may rewrite this as 
\begin{align*}
  \sum_{x\in X_0} \frac{1}{(n-\sum x_i)!}\varphi_\lambda(x)&
  =\sum_{\stackrel{\sum m_i^2=\lambda}{M=\{m_1\le\dots\le m_n\}}} a_M(-1)^{\sum m_i}
  \left[\prod_{i=1}^n (1+(-1)^{m_i}) -(-1)^{\sum m_i}-1\right].
\end{align*}
But at least one $m_i$ is odd, hence the product in the bracket is $0$. Furthermore, the sum of $m_i$ is also odd, hence the whole bracket is $0$. Thus

\begin{align*}
  \sum_{x\in X_0} \frac{1}{(n-\sum x_i)!}\varphi_\lambda(x)=0.
\end{align*}

Therefore either $\varphi_\lambda$ is $0$ at the centers of faces of arbitrary dimension, or $\varphi_\lambda$ must change sign. To prove the eigenfunction must change sign we will use a different method, similar to the one for equilateral triangles (\autoref{sec:trieignonneg}).  For the moment, we can show that eigenfunction cannot be positive on the boundary for a few low-dimensional cases with an argument about the parity of the $m_i$.

\begin{proposition}\label{lowdim}
  In dimensions 2, 3 and 4, any positive Neumann eigenfunction on a cube must be constant.
\end{proposition}
\begin{remark}
  Dimension 2 was proved by Hoffman-Ostenhof \cite{HO13}. 
\end{remark}
\begin{proof}
  We only need to consider even $m$.
For $0<h<1$ define   
\begin{align*}
  X_h=\{(1=x_1\ge x_2\ge\dots\ge x_n=h):x_i\in\{h,1\}\}.
\end{align*}
As above we get
\begin{align*}
    \sum_{x\in X_h} &\frac{1}{(n-\sum 1_{x_i}(1))!}\varphi_\lambda(x)
    \\&\qquad\qquad=\sum_{\stackrel{\sum m_i^2=\lambda}{M=\{m_1\le\dots\le m_n\}}} a_M
  \left[\prod_{i=1}^n (\cos(m_i\pi h)+(-1)^{m_i}) -1-\prod_i \cos(m_i\pi h)\right],
\end{align*}

We now consider each dimension separately: 
\begin{itemize}
    \item \textit{Dimension 2:}

The sum of the squares of two odd numbers is congruent to $2$ modulo $4$, hence it is not a square. Therefore, if $m$ is even, then both $m_i$ are even. Furthermore, by induction $2^s$ divides both $m_i$, but $2^{s+1}$ divides exactly one of them. Take $h=1/2^s$. Then $\cos(m_i\pi/2^s)$ have both signs. But both $m_i$ are even, hence the first product in the bracket is $0$, and the second product equals $-1$. Hence the whole bracket equals $0$.

\item \textit{Dimension 3:}

  The sum of three squares is again a square only if all numbers are even. Indeed, with two odd numbers, the sum that is congruent to $2$ modulo $4$. By induction, there exists $s$ such that $2^s$ divides all $m_i$, while $2^{s+1}$ divides none or two. In either case, the first product equals $0$, and the second equals $-1$. Hence the bracket is again $0$.

\item \textit{Dimension 4:}
  
The sum of $k$ odd squares is congruent to $k$ modulo $8$. Hence only $1$ or $4$ odd squares can give a square. Suppose some $m_i$ are odd. Since $m$ is even, all $m_i$ must be odd and $4$ does not divide $m$. Since all $m_i$ are odd, $\cos(m_i\pi/2)=0$ and the first product equals $1$. The second product is obviously $0$ and the bracket is again $0$. If all $m_i$ are even, but $4$ does not divide $m$, then exactly one of the $m_i/2$ is odd. Therefore the first product is $0$ and the second equals $-1$. Again the bracket is $0$. Finally, suppose $4$ divides $m$. Then $4$ divides all $m_i$, and we can reduce the problem to $m'=m/4$ and apply the same argument recursively.
\end{itemize}
\end{proof}

\begin{remark}
  In dimension 5 we have $36=6^2=4\times 3^2$. The first decomposition does give $0$ in the bracket. However the second gives $1$.

  In dimension 6 we have $36=6^2=2\times 4^2+4\times 1^2=5^2+2\times 2^2+3\times 1^2=5^2+3^2+2\times1^2$. Hence in dimensions 6 and higher, any integer smaller than $m$ may appear in the decomposition for $m^2$. Therefore an argument based on divisibility will most likely fail.
\end{remark}

\subsection{Nonnegative eigenfunctions}
To prove \autoref{thm:cubes} we will generalize the approach used on equilateral triangles in \autoref{sec:trieignonneg}. We will show that $\varphi_\lambda$ integrates to $0$ over a union of small cubes with codimension one contained in one of the faces. Since an eigenfunction cannot equal $0$ on an open subset of the boundary (it already satisfies the Neumann condition there), it must change sign in the union of these cubes. Note also that it is irrelevant if these cubes are disjoint, but they must be subsets of the face.

Suppose that $\lambda=m^2$ and consider the following set of points uniformly distributed on $(-1,1)$.
\begin{align*}
  X= \left\{x_k=1-\frac{2k+1}{m}:\;k=0,\dots,m-1\right\}
\end{align*}
By \cite[Section 1.341, Formula 3]{GR00}
 \begin{align*}
   \sum_{k=0}^{m-1} \cos(l\pi x_k)=
   \begin{cases}
     0& 0<l<m,\\
     m (-1)^{m+1}&l=m.
   \end{cases}
 \end{align*}
 Now take a lattice of cubes with centers on $X^{n-1}$ and side length $2a$. That is 
   \begin{align*}
     \mathcal L=\left\{ C_x=\{y:y_n=1, \max|x_i-y_i|\le a\}: x\in X^{n-1} \right\}
   \end{align*}
 Note that all cubes $C_x$ are on one face of $[-1,1]^n$ if $a<1/m$. 
 
 Consider one sequence $m_i$ and one permutation in the definition of $\varphi_\lambda$. The integral over the lattice of the resulting function equals
 \begin{align}
   \sum_{C_x\in \mathcal L} &\int_{C_x} \cos(m_{\sigma(n)} \pi) \prod_{i=1}^{n-1} \cos(m_{\sigma(i)} \pi z_i)dz_1\dots dz_{n-1}=\nonumber
   \\&=
   \cos(m_{\sigma(n)}\pi)\prod_{i=1}^{n-1} \sum_{k=0}^{m-1} \int_{x_k-a}^{x_k+a} \cos(m_{\sigma(i)}\pi z_i) dz_i=\nonumber
   \\&=
   \cos(m_{\sigma(n)}\pi)\prod_{i=1}^{n-1} \sum_{k=0}^{m-1} \frac{2}{m_{\sigma(i)}\pi}\sin(m_{\sigma(i)}\pi a)\cos(m_{\sigma(i)}\pi x_k)=\label{sinx/x}
   \\&=
   \begin{cases}\label{eq:cases}
     0,& (m_1,\dots,m_n)\ne (0,\dots,0,m),\\
     (-1)^m (2am)^{n-1},& m_{\sigma(n)}=m,\\
     \frac1\pi\sin(m\pi a)(-1)^{m+1}(2am)^{n-2},& m_{\sigma(i)}=m \text{ for some } i<n.
   \end{cases}
 \end{align}
 Note that in \eqref{sinx/x} we mean $\frac{\sin x}{x}=1$ if $x=0$. The top case in \eqref{eq:cases} is equivalent to $k>0$ in \autoref{sec:trieignonneg}, while the other two cases correspond to the integral from \eqref{MMmode}.

Hence
\begin{align*}
    \sum_{C_x\in \mathcal L} \int_{C_x} \varphi_\lambda(z) dz &= 2a_{0,\dots,0,m} (2am)^{n-2}\left(\sum_{\sigma(n)=n}am-\sum_{\sigma(n)\ne n} \frac{1}{\pi}\sin(m\pi a)\right)=
  \\&=
2a_{0,\dots,0,m} (2am)^{n-2}(n-1)!\left(am-\frac{n-1}{\pi}\sin(m\pi a)\right).
\end{align*}
The last expression equals $0$ if we choose $0<a<1/m$ so that 
\begin{align*}
  \pi a m = (n-1)\sin(\pi a m).
\end{align*}
The existence of such $a$ is equivalent to the existence of $0<x<\pi$ such that
\begin{align}\label{sineq}
  x=(n-1)\sin x.
\end{align}
This equation has a positive solution when $n>2$. This proves that in dimensions $n>2$ any eigenfunction of a cube must change sign on the boundary. However this argument fails in dimension 2, and \autoref{lowdim} (or the earlier result~\cite{HO13} by Hoffman-Ostenhof) is the best we can expect. It is remarkable that \eqref{sineq} is exactly the same as the equation for $a$ the the equilateral case (perhaps hinting at the fact that the equilateral triangle can be embedded in a cube as an intersection of that cube with a plane).

\subsection{General boxes}
Consider an $n$-dimensional box with sides of length $2a_i$ centered at the origin. The eigenvalues $\lambda$ for this box can be indexed using a sequence $L$ of $n$ natural numbers $l_i$ such that 
\begin{equation}
\label{eigenvalue as sum over L}
  \lambda=\frac{\pi^2}{4}\sum_{i=1}^n \frac{l_i^2}{a_i^2}.
\end{equation}
The complete set of eigenfunctions is given by
\begin{align*}
  \varphi(x)=\prod_{i=1}^n F(l_i \pi x_i/a_i),
\end{align*}
where $F$ is either sine or cosine. However, any eigenfunction that is nonnegative on the boundary can be axially symmetrized by summing over all sign changes for all coordinates. This procedure still gives a nonnegative boundary and eliminates all occurrences of sine. Therefore we can assume that $F(x)=\cos x$.

Any eigenfunction of a box, when restricted to a face, is also a sum of eigenfunctions on each face (put $x_i=a_i$ for some $i$). This lower dimensional combination of eigenfunctions consists of eigenfunctions that are orthogonal to a constant eigenfunction (that is, they integrate to $0$ over the face), and/or a constant term. If the constant term is not present, the linear combination must change sign on the face. 
Therefore, an eigenfunction that is nonnegative on the boundary must have a constant term when restricted to any face. Hence, its eigenvalue must admit indexing sequences $L_j=\{l_i=\delta_j(i)l_j\}$.  Taking $L = L_j$ in Equation~\eqref{eigenvalue as sum over L} thus yields
\begin{align*}
  \lambda=\frac{\pi^2}{4}\frac{l_1^2}{a_1^2}=\frac{\pi^2}{4}\frac{l_2^2}{a_2^2}=\dots=\frac{\pi^2}{4}\frac{l_n^2}{a_n^2}.
\end{align*}

This immediately proves that if any ratio of two squares of the side lengths is not the square of a rational number, then nonnegative eigenfunctions do not exist. 

For any $i\ne j$ we have
\begin{align*}
  \frac{a_i^2}{a_j^2}=\frac{l_i^2}{l_j^2}.
\end{align*}
Hence $a_i/a_j$ is also rational for any $i\ne j$. Therefore $a_i=r_i\alpha$ for some rational $r_i$ and real $\alpha$, and this box can be used to tile a cube. Then any eigenfunction positive on the boundary of this box gives an eigenfunction on a cube with the same property (thanks to Neumann boundary matching in the tiling). But we proved these do not exist. Therefore \autoref{thm:cubes} also holds for arbitrary boxes.

\section{Proof of \autoref{thm:regular} for $n\ge 6$}\label{sec:nge6}
    For a regular hexagon we can simply take the symmetric mode $\varphi_{0,1}$ of the equilateral triangle (defined in \eqref{eq:symeigf}) and cover the hexagon with its reflections to get an eigenfunction which is positive on the boundary.

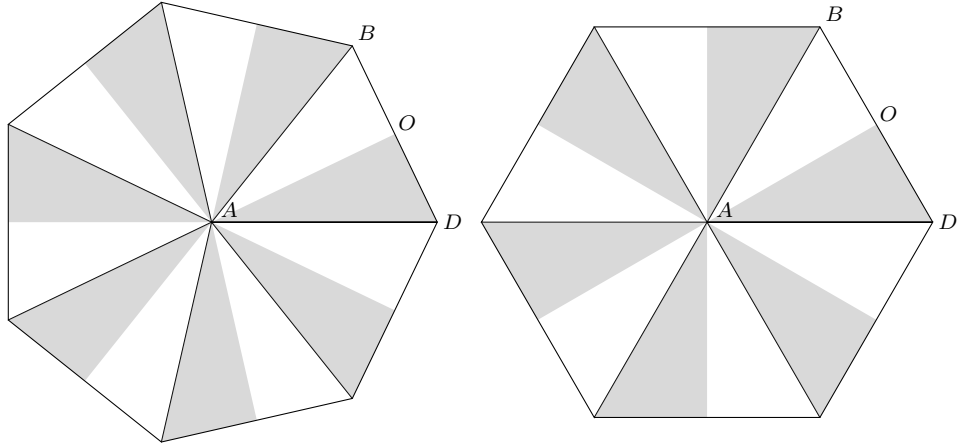
\begin{figure}[t]
  \begin{center}
\begin{tikzpicture}[scale=3,baseline=0]
  \regularlabelled{7}
\end{tikzpicture}
\begin{tikzpicture}[scale=3,baseline=0]
  \regularlabelled{6}
\end{tikzpicture}
  \end{center}
  \caption{Regular heptagon decomposed into \textbf{subequilateral} triangles, and regular hexagon decomposed into equilateral triangles.}
  \label{fig:seven}
\end{figure}

    Now consider a regular polygon with $n$ sides, where $n>6$.  Such a polygon can be decomposed into $n$ subequilateral triangles ($ABD$ on \autoref{fig:seven}). The second Neumann eigenvalue of a subequilateral triangle is simple (\autoref{simple}) and the second Neumann eigenfunction is symmetric \cite[Theorem 3.1]{LSminN}. Hence it is also the second eigenfunction on the right triangle formed by cutting the isosceles triangle in half ($ABO$ and $ADO$ on \autoref{fig:seven}).

    The second Neumann eigenfunction must have exactly 2 nodal domains, by Courant's nodal domain theorem (see e.g. \cite[Sec. V.5, VI.6]{CH1}). By symmetry, the nodal line must either connect the two long sides ($AB$ and $AD$) of the subequilateral triangle, or start and end on the short side ($BD$). From \autoref{line}, the second case is not possible, regardless of the shape of the triangle. Hence this eigenfunction is positive on the short side, and it can be reflected $n$ times inside of the regular polygon to cover the whole regular polygon. We obtain an eigenfunction on the regular polygon that is positive on the boundary. Therefore \autoref{thm:regular} holds for $n> 6$. 
    
As a corollary from the above proof we also get a partial result for \autoref{con:nodal} 
\begin{corollary}
The nodal line for the second Neumann eigenfunction on right triangles with smallest angle $\alpha<\pi/6$ connects the interiors of the two longest sides.
\end{corollary}

\section{Proof of \autoref{thm:regular} for regular pentagons.}\label{sec:pentagon}

A regular pentagon decomposes into acute superequilateral triangles instead of subequilateral triangles (as was the case of $n\ge 6$ sides). 

The second Neumann eigenvalue $\mu_2$ of a superequilateral triangle ($ABD$ on \autoref{fig:five}) is simple but the second eigenfunction is antisymmetric \cite[Theorem 3.2]{LSminN} (as opposed to symmetric for subequilateral triangles). By \autoref{symmetric} all eigenfunctions for $\mu_3$ are therefore symmetric. But all these eigenfunctions belong to the second (simple) eigenvalue of the right triangle $OAB$ obtained by cutting the isosceles triangle $ABD$ in half (shaded on \autoref{fig:five}). Therefore $\mu_3$ of a superequilateral triangle $ABD$ is simple, with the eigenfunction symmetric with respect to $OA$. Unfortunately, \autoref{line} applies only to eigenfunctions for $\mu_2$. Moreover, we need to exclude a possibility of having 3 nodal domains (allowed for $\mu_3$). 

\begin{figure}[t]
  \begin{center}
\begin{tikzpicture}[scale=3]
\regular{5}
  \fill ({-cos(36)},0) coordinate (0) circle (0.02) node [below right=-2pt] {\tiny $O$}; 
  \draw[blue,very thick] (0) -- (0,0) coordinate (a) -- (144:1) coordinate (b) -- (0) node [pos=0.5,left=-1pt] {\tiny $T$};
  \draw[red] (0,0) -- (144:1) -- (-{2*cos(36)},0) coordinate (c) node[pos=0.7,above] {\tiny $R$} -- (-144:1) coordinate (d) -- cycle;
  \fill (a) circle (0.02) node [below right=-2pt] {\tiny $A$}; 
  \fill (b) circle (0.02) node [above=-2pt] {\tiny $B$}; 
  \fill (c) circle (0.02) node [below left=-2pt] {\tiny $C$}; 
  \fill (d) circle (0.02) node [below=-2pt] {\tiny $D$}; 
\end{tikzpicture}
  \end{center}
  \caption{Regular pentagon decomposed into acute \textbf{superequilateral} triangles, blue triangle $T$ and red rhombus $R$.}
  \label{fig:five}
\end{figure}
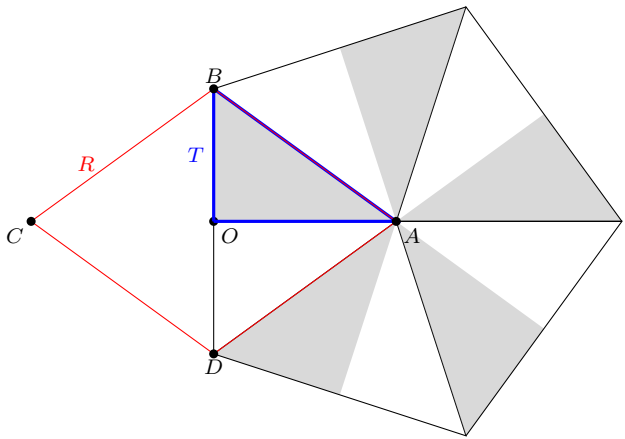

These two problems make the pentagonal case much harder than regular polygons with $n\ge 6$ sides.  Moreover, there is essentially no hope of finding explicit trigonometric formulae for its eigenfunctions.  A completely different approach is required.

Consider the rhombus $R$ ($ABCD$ on \autoref{fig:five}) built using right triangle $T$ ($ABD$ on the same figure) with the smallest angle at least $\pi/6$ (equal to $\pi/5$ for our regular pentagon). Then \cite[Corollary 1.3]{Srhombi} gives
\begin{align*}
  \mu_4(R)<\lambda_1(R).
\end{align*}
Note that the classical Levine-Weinberger inequality \cite{LW86} only gives $\mu_3\le\lambda_1$. Furthermore, the eigenfunction $u_2$ that belongs to $\mu_2(T)$ extends by symmetry to a doubly symmetric eigenfunction $\tilde u$ on $R$. Then $\tilde u$ must belong to the lowest eigenvalue of $R$ which possesses a doubly symmetric mode. Otherwise, a doubly symmetric eigenfunction of the lower eigenvalue of $R$ would be an eigenfunction for $T$. Therefore
\begin{align*}
  \mu_2(T)=\mu_4(R).
\end{align*}
For the triangle $T=ABD$ we have
\begin{lemma}\label{partials}
  The partial derivatives $u_x$ and $u_y$ of the second Neumann eigenfunction $u$ of $T$ are never zero and have opposite signs.
\end{lemma}
\begin{remark}
    Note that this result can be deduced from the last paragraph on page 244 of Atar-Burdzy~\cite{AB04}. Nevertheless we present a simpler proof.
\end{remark}

\begin{proof}
We will follow the proofs of \cite[Lemmas 3.2,3.4]{Mi12} and \cite[Theorem 2]{S13}. First note that \cite[Lemma 2]{S13} applies to $T$, hence its second Neumann eigenfunction $u$ is strictly monotonic on $AB$. We can assume that $u_x>0$ and $u_y<0$ on $AB$. 

Now we consider the doubly symmetric extension of $u$ to the rhombus $R$. On $CB$ we also have $u_y<0$ due to double symmetry of $u$, while on $CD$ and $DA$ we have $u_y>0$. Similarly, $u_x>0$ on $DA$, and $u_x<0$ on $CB$ and $CD$. Furthermore, $u_x$ is antisymmetric with respect to $y$-axis and symmetric with respect to $x$-axis (again by double symmetry of $u$), while $u_y$ has reversed symmetries.

Suppose $u_y$ is zero somewhere in $R$, then by \cite[Proposition 2.1(i)]{Mi12} it must change sign inside $R$. By antisymmetry, it must be positive somewhere in $ABC$. But $u_y<0$ on $AB$ and $CB$ and $u_y=0$ on $AC$. Hence a nodal domain of $u_y$ is a subset of $ABC$ (part of the nodal line might be a subset of $AC$). We already noticed that $u$ belongs to $\mu_4(R)$, hence
\begin{align*}
  \mu_4(R)=\lambda_1(N)>\lambda_1(R)>\mu_4(R),
\end{align*}
a contradiction. Hence $u_y<0$ on $ABC$ (hence also on $T$). Similarly we can prove that $u_x>0$ on $T$.
\end{proof}

We need a domain monotonicity result for the eigenvalues of the domains with mixed boundary conditions. This is a special case of a more general partial domain monotonicity principle proved by Harrell. We present this special case due to its rather simple proof.

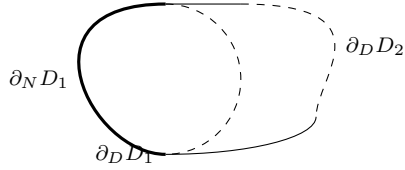
\begin{figure}[t]
  \begin{center}
    \begin{tikzpicture}
      \draw[very thick] (0,0) .. controls (-1,0) and (-2,2) .. (0,2) node [left,pos=0.5] {\tiny $\partial_N D_1$};
      \draw[dashed] (0,2) arc (90:-90:1) node [left,pos=0.5] {\tiny $\partial_D D_1$}; 
      \draw (0,2) -- (1,2);
      \draw (0,0) arc (270:360:2 and 0.5) coordinate (a);
      \draw[dashed] (1,2) .. controls (3,2) and +(0,0.5) .. (a) node [right,pos=0.5] {\tiny $\partial_D D_2$};
    \end{tikzpicture}
  \end{center}
  \caption{Domain monotonicity: $D_1\subset D_2$ and Neumann boundary condition on $D_1$ is specified only on a portion of the Neumann boundary $\partial_N D_2$. Solid lines indicate a Neumann boundary while dashed lines indicate Dirichlet boundarh.}
  \label{fig:monotonicity}
\end{figure}
\begin{lemma}[\protect{Special case of Harrell \cite[Corollary II.2]{Ha06}}]\label{lem:harrell}
  Suppose $D_1\subset D_2$ are open and the Neumann boundary $\partial_N D_1$ of $D_1$ is contained in the Neumann boundary $\partial_N D_2$ of $D_2$ (see \autoref{fig:monotonicity}). Then the lowest eigenvalue on $D_2$ for the mixed Dirichlet-Neumann problem is smaller than the lowest mixed eigenvalue on $D_1$, unless $D_1=D_2$ and $\partial_N D_1=\partial_N D_2$. 
\end{lemma}
\begin{proof}
  Suppose $\varphi$ is the eigenfunction for $D_1$. Extend it with $0$ to the whole set $D_2$. 
  Note that the extension satisfies the Dirichlet boundary condition on $\partial D_2\setminus \partial_N D_1$, hence on the Dirichlet boundary $\partial_D D_2$. Note also that $\partial_N D_1$ does not intersect $D_2\setminus D_1$, hence the extension is continuous.
  Therefore it is a valid trial function for the Rayleigh-Ritz quotient on $D_2$. But it also equals $0$ on an open set (if $D_1$ strictly included in $D_2$), or equals $0$ on a piece of boundary $\partial_N D_2\setminus \partial_N D_1$ (satisfies both Dirichlet and Neumann condition). In either case it must be $0$ everywhere.
\end{proof}

\begin{corollary}\label{nodal line}
  Let $D_1\subset D$ be a nodal domain for the eigenfunction for $\mu_2(D)$. Suppose we find $D_2\subset D$ such that $\partial_N D_1\subset \partial_N D_2\subset \partial D$, and the mixed eigenvalue $\lambda_1(D_2)> \mu_2(D)$. Then the nodal line for $\mu_2(D)$ must intersect the Dirichlet boundary $\partial_D D_2$.
\end{corollary}
\begin{proof}
  The mixed eigenvalue of $D_1$ equals $\mu_2(D)$ and is smaller than $\lambda_1(D_2)$. Note also that for a nodal domain we always have $\partial_N D_1\subset \partial D$ and $\partial_D D_1$ is the nodal line. If $D_1\subset D_2$, then the above lemma gives $\lambda_1(D_1)>\lambda_1(D_2)$, leading to a contradiction.  Hence $D_1\not\subset D_2$, and the nodal line $\partial_D D_1$ must have a nonempty intersection with $D\setminus D_2$. Hence it must intersect $\partial_D D_2$.
\end{proof}

\begin{lemma}
  The lowest eigenvalue of a mixed Dirichlet-Neumann eigenvalue problem is simple and the eigenfunction can be taken positive inside of the domain.
\end{lemma}
\begin{proof}
Let $u$ be the eigenfunction for the smallest mixed eigenvalue. Then $|u|$ gives the same value of the Rayleigh-Ritz quotient. This means that $|u|$ is also a minimizer of the Rayleigh-Ritz quotient. Any minimizer is an eigenfunction (see e.g. \cite[page 55]{Lnotes}).
Then elliptic regularity shows that $|u|$ solves the eigenvalue problem classically (not in the weak sense). But $\Delta |u|=-\lambda|u|\le 0$, so $|u|$ is superharmonic. Take any point $p$ inside the domain at which $|u|=0$.  Nonnegativity at $p$ violates the minimum principle. Therefore $u>0$ inside of the domain.

Suppose $u> 0$ and $v$ are orthogonal eigenfunctions belonging to the lowest eigenvalue. Orthogonality forces $v$ to change sign in the domain, and the argument above shows that $|v|$ is yet another eigenfunction, violating the minimum principle. Therefore the eigenvalue is simple.
\end{proof}

Now we restrict our attention to the right triangle $T=OAB$ with angle $\pi/5$ near $A$ (see \autoref{fig:five}). We can assume that $|OA|=1$ and the second Neumann eigenfunction $u$ is negative at $B$. We know that $u_x$ is positive and $u_y$ is negative by \autoref{partials}. This implies:
\begin{lemma}\label{le:monotonic}
The nodal line for the eigenfunction for the lowest positive eigenvalue on the right triangle $OAB$ is the graph of a strictly increasing function, hence it must touch the longest side $AB$. 
\end{lemma}

Suppose that the nodal line also touches the shortest side $OB$ (or hits the origin $O$).  We construct an explicit subdomain $D_2$ of $OAB$, such that a nodal domain of the second eigenfunction of T is contained in $D_2$, yet $\lambda_1(D_2)>\mu_2(T)$, contradicting \autoref{nodal line}.

Therefore the second eigenfunction of $T$ has a fixed sign on $OB$. Reflecting this eigenfunction 10 times we can cover a regular pentagon and find an eigenfunction which is negative on its boundary. This proves the remaining case $n=5$ in \autoref{thm:regular}.

 \begin{lemma}\label{lem:stretching}
  Let $\Omega$ be a Lipschitz domain and $\Omega_t=\{(x,ty): (x,y)\in\Omega\}$ (a stretched domain), for $t>1$. Then for any mixed boundary conditions $\lambda_1(\Omega_t)\le\lambda_1(\Omega)$. The same is true for the smallest nonzero Neumann eigenvalues.
\end{lemma}
\begin{proof}
  See the last paragraph on page 132 of \cite{LSminN}. 
\end{proof}

To avoid irrational triangle vertices, we define right triangles $T'\subset T\subset T''$, as well as a reference triangle $T_{\text{ref}}$.  These triangles have two common vertices $O=(0,0)$ and $A=(1,0)$, and their third vertices are as follows:
\begin{align}
    T':&\, B'=(0, 85/117) = (0, 0.72649573\ldots) ,\label{eq:Tp}\\
    T:&\, B=(0, \tan(\pi/5) = (0, 0.72654252\ldots),\label{eq:T}\\
    T'':&\, B''=(0, 93/128) = (0, 0.7265625), \label{eq:Tpp}\\
    T_{\text{ref}}:&\, B_{\text{ref}}=(0,1).
\end{align}
The triangle $T$ can be used to create the regular pentagon. Let us define linear maps 
\begin{align}
    \varphi:\, T_{\text{ref}}&\to T,\\
    \psi:\, T_{\text{ref}}&\to T''
\end{align}
which fix $A$ and send $B_{\text{ref}}$ to $B$ and $B''$, respsectively.

\begin{lemma}\label{lem:threshold}
    The smallest nonzero  Neumann eigenvalues of $T$ and $T'$ satisfy
    \begin{align*}
        \mu_T\le \mu_{T'}< 12.25.
    \end{align*}
\end{lemma}
\begin{proof}

    The triangle $T$ is obtained by vertically stretching $T'$, hence \autoref{lem:harrell} gives the required monotonicity.  We can get a very accurate upper bound for $\mu_{T'}$ using the finite element method with quadratic nodal conforming elements. An eigenfunction approximation obtained from conforming elements is a valid trial function (continuous, piecewise quadratic) for the Rayleigh-Ritz quotient, hence we get a strict upper bound for $\mu_{T'}$ by plugging it directly into the Rayleigh-Ritz quotient. We tesselate $T'$ by congruent triangles with $N$ distinct nodes, and construct test functions in $H^1(T')$ which are continuous and piecewise quadratic. Standard finite element approximation arguments tell us that these Rayleigh-Ritz quotients will form a decreasing sequence in $N$, which converges to the true $\mu(T')$. In \autoref{tab:UppermuT'} we show the Rayleigh-Ritz quotient computed using interval arithmetic (we only report 4 digits, though the interval around these approximations is of width $1e-9$).  We can comfortably bound $\mu(T')$ above:
    \begin{table}[t]
    
	\begin{tabular}{@{}cc@{}}
	    \toprule
    Number of nodes N & Rayleigh Quotient\\
	    \midrule
    38&12.2482\\
    128&12.2476\\
    463&12.2475\\
	    \bottomrule
    \end{tabular}
    \caption{Upper bounds for $\mu_{T'}$ using conforming quadratic finite elements (validated numerics.}
    \label{tab:UppermuT'}
    \end{table}

\begin{align*}
  \mu_{T'} \le 12.2483< 12.25.
\end{align*}
\end{proof}

In what follows we will use a threshold value $V=12.25$. We choose this particular value to make sure $\mu_T<V$.

\begin{lemma}\label{lem:subdomainmonotonicity}
    Let $D\subset T_{\text{ref}}$. Clearly $\psi(D)\subset T'', \varphi(D)\subset T$. Impose any mixed Dirichlet-Neumann boundary conditions on $\partial D$. Then $\psi$ and $\varphi$ induce mixed boundary conditions on $\psi(\partial D)$ and $\varphi(\partial D)$. The smallest eigenvalues for the mixed eigenvalue problems on $\varphi(D)$ and $\phi(D)$ then satisfy
    \begin{align*}
	\lambda(\varphi(D))\ge \lambda(\psi(D)).
    \end{align*}
\end{lemma}
\begin{proof}
    The domain $\psi(D)$ is a vertically stretched version $\varphi(D)$. Therefore \autoref{lem:stretching} implies the result.
\end{proof}

\begin{definition}[cf. \autoref{fig:alg}]\label{def:pUL}
    Let $a=1$ and $b=\tan(\pi/5)$ and $N=64$. Define a grid on $T$:
\[
G = \{(ai/N, bj/N)\;:0 \leq i,j \leq N, i+j \leq N\;\},
\]
and for a grid point $p$ let
\begin{align*}
R_L(p) &= \{(x, y) \;|\; x \geq x_0,  y \leq y_0\} \cap T, \\
R_U(p) &= \{(x, y) \;|\;  x \leq x_0, y \geq y_0\} \cap T. 
\end{align*}
Furthermore define two sequences of points:
\begin{align}
    \begin{split}
	p_L=\big\{(28,21),\,(31,27),\,(24,14),\,&(28,23),\,(26,19),\\
	      &(21,9),\,(30,27),\,(23,15),\,(31,30)  \big\},
    \end{split}\label{eq:pL}\\
\begin{split}
    p_U=\big\{(25,24),\,(20,17),\,(28,30),\,&(18,13),\,(22,18),\,(28,28),\\
		&(29,29),\,(19,12),\,(27,25),\,(25,20) \big\}.
\end{split}\label{eq:pU}
\end{align}
\end{definition} 
For these particular points $p_L^{(k)}$ and $p_{U}^{(k)}$ we define sequences of subdomains of $T$.
\begin{definition}\label{def:DUL}
    Let $U^{(0)}=L^{(0)}=\emptyset$, and
\begin{align*}
L^{(k+1)} &= L^{k} \cup R_L(p_L^{(k+1)}),\\
U^{(k+1)} &= U^{k} \cup R_U(p_U^{(k+1)}).
\end{align*}
We call these sets the \emph{upper} and \emph{lower exclusion regions} respectively, for reasons which will become clear shortly.  Also define the following subdomains of $T$ with associated mixed boundary conditions:
\begin{align}
    D_U(k,p) &= T\setminus (L^{(k)}\cup R_L(p)),\,\,\text{Dirichlet on } \partial D_U(k,p)\setminus(\overline{AB}\cup \overline{OB}), \text{Neumann elsewhere}\\
    D_L(k,p) &= T\setminus (U^{(k)}\cup R_U(p)),\quad\text{Dirichlet on } \partial D_L(k,p)\setminus \partial T, \text{Neumann elsewhere}
\end{align}
\end{definition}
\begin{remark}
We note that for appropriate choices of points $p=p_L^{(k+1)}, p_U^{(k+1)}$ in \autoref{eq:pL} and \autoref{eq:pU} will make 
$$ D_U(k,p_L^{(k+1)})=  T\setminus (L^{(k+1)}), \qquad D_L(k,p_U^{(k+1)}) = T\setminus(U^{(k+1)}).$$
 We use will this fact later on.
\end{remark}

\begin{lemma}\label{lem:validated}
    The lowest mixed eigenvalues on the following subdomains of $T$ and $T''=\psi\varphi^{-1}T$ satisfy
    \begin{align}
	\lambda_1\left(D_U\left(k, p_U^{(k+1)}\right)\right)&\ge \lambda_1\left( \psi\varphi^{-1}D_U\left(k,p_U^{(k+1)}\right) \right)>V,\qquad\text{ for }0\le k\le 9,\\
	\lambda_1\left(D_L\left(k+1, p_L^{(k+1)}\right)\right)&\ge \lambda_1\left( \psi\varphi^{-1}D_L\left(k+1, p_L^{(k+1)}\right) \right)>V,\qquad\text{ for }0\le k \le 8,\\
	\lambda_1\left(D_L(10, (0,0))\right)&\ge \lambda_1\left( \psi\varphi^{-1}D_L(10, (0,0)) \right)>V.
    \end{align}
\end{lemma}
\begin{remark}
    Note that in the last case we use $p = (0,0)$, which results in no enlargement to $U^{(10)}$. Effectively, $D_L(10, (0,0))=T\setminus U^{(10)}$.
\end{remark}

The proof is postponed to \autoref{Section:Numerics}. The proof involves lower bounds for the eigenvalues of large sparse rational matrices.

\begin{lemma}\label{lem:nodalline}
    Assuming the nodal line for $\mu_T$ starts on the side $OB$, it does not intersect any of the excluded sets $U^{(k)}$ and $L^{(k)}$ (e.g. the shaded regions on \autoref{fig:alg}).
\end{lemma}

\begin{proof}
    The claim is clearly true for $U^{(0)}$ and $L^{(0)}$. Assume the claim holds for $U^{(k)}$ and $L^{(k)}$. We construct $D_U(k, p_U^{(k+1)})$ and note that $\lambda_1(D_U(k, p_U^{(k+1)}))$ is larger than the threshold $V$ (by \autoref{lem:validated}). Hence $D_U(k, p_U^{(k+1)})$ is a subdomain of $T$ whose lowest eigenvalue exceeds the smallest nonzero Neumann eigenvalue $\mu_T$ for T. Therefore \autoref{nodal line} implies that the nodal line for the eigenvalue $\mu_T$ must intersect $D_U(k, p_U^{(k+1)})$. This cannot happen on $L^{(k)}$, and hence the nodal line intersects $R_L(p_U^{(k+1)})\setminus L^{(k)}$. Monotonicity of the nodal line (\autoref{le:monotonic}) ensures that point $p_L^{(k+1)}$ is separated from $L^{(k)}$ by the nodal line. In particular it must belong to a different nodal domain than $L^{(k)}$, as does any point above and to the left of $p_L^{(k+1)}$. Hence the nodal line cannot intersect the constructed $U^{(k+1)}$. 

A similar argument applies to $L^{(k+1)}$.
  
\end{proof}

\begin{theorem}\label{thm:twolongestsides}
    The nodal line for the eigenfunction for $\mu_T$ does not intersect $OB$.
\end{theorem}

\begin{proof}
    Suppose the nodal line has its endpoints on $AB$ and $OB$ (\autoref{le:monotonic} ensures it ends on $AB$).
    Divide $T$ into two nodal domains of the eigenfunction of $\mu_T$.
    By construction, the smallest mixed eigenvalues of these nodal domains equal $\mu_T$, which is not larger than the threshold $V$. Furthermore,one of these nodal domains is contained in $ T\setminus U^{(10)}$.
    
    Domain monotonicity (\autoref{lem:harrell}) and the last case of \autoref{lem:validated} lead to a contradiction.
\end{proof}
This theorem ensures that the eigenfunction for $\mu_T$ of the right triangle $T$ does not change sign on the shortest side. Its symmetric extension to the regular pentagon is an eigenfunction of that pentagon, which does not change sign on the boundary, proving the remaining case of \autoref{thm:regular}.

\section{The exclusion search algorithm}\label{sec:algorithm}
Our proof for pentagons contains a rather strange collection of domains and associated grid points $p_L$ and $p_U$ defined in \eqref{eq:pL} and \eqref{eq:pU}. 
In \autoref{lem:validated} we check that the smallest mixed eigenvalues of these domains satisfy appropriate inequalities. The lemma does not depend on, or explain, how these exact domains were chosen. In reality, these domains were found by a computer search.  In this section, we present an algorithm which generates these points $p_L$ and $p_U$ and the associated domains based on numerical computations. In the following section we describe how, for a given domain, we validate the assertion of \autoref{lem:validated}.

The proof in the preceding section is a human-readable, computer-assisted proof, and it is worth making a comment about the structure of such proofs.  Of course, all mathematicians appreciate a short, beautiful proof of a theorem, but it is not clear what standard of ``beauty'' should apply to a computer's work.  In this section, we argue that a computer proof should strive to have \emph{small certificates}, be \emph{easily checkable}, and be \emph{adaptable to other situations}, and that our proof has these properties.  We direct the interested reader to~\cite[Ch. 2]{aeqb} for a reference on human-readable, computer-assisted proofs.  

\subsection{Small Certificate}
The computer-generated part of a human-readable proof is called the \emph{proof certificate}.
Because computers can generate a lot of output very quickly, the proof certificates that they generate can be very long and time-consuming to read, unless care is taken to program the computer to search for \emph{short} proofs.   

For example, Appel and Haken's famous proof of the four-color theorem in graph theory~\cite{AH77a,AH77b,AH89} is partly computer-generated; its certificate consists of a certain argument by cases.  This portion of the argument is very long: it occuped a lengthy microfiche supplement to the articles~\cite{AH77a,AH77b}.  A later version of the proof appeared in the book~\cite{AH89}, which is over 700 pages long.  A mathematician who wanted to read their proof would, amongst other tasks, have to inspect all of this work carefully; it is likely that this would take years of full-time work to do.  More recent research on the four-color theorem, such as~\cite{robertson-etal-1997}, has in large part focused on finding computer-generated proofs with shorter, simpler proof certificates.  These proofs are much less time-consuming for a mathematician to read.

We found a rather short proof certificate for our problem: it is the two lists of points \eqref{eq:pL}, \eqref{eq:pU}, together with the upper bound $\mu_{T'} < 12.2483$ in the proof of Lemma 6.7.

\subsection{Checkability}
When possible, a computer-assisted proof should be checkable without requiring specialized software.
To check our proof, a human needs only use standard, validated numerics sofware (such as {\tt INTLAB}~\cite{Ru99a}, for interval arithmetic) to give bounds on the eigenvalues of certain matrices which we describe.  An interested reader can use this, or other standard interval arithmetic software, to check the proof without needing to develop their own specialized code. 

\subsection{Adaptability}
A mathematician should be able to adapt a computer-assisted proof to other problems, without requiring significant changes in the human-written part.
Our proof has this advantage: it is possible to change the grid size $N=64$, or triangle height $b=\tan(\pi/5)$, and re-run our algorithm.  This would produce a new list of points $p_L$ and $p_U$, and a new upper bound for $T'$.  Having made these changes, the reader would select a new threshold $V$ exceeding this upper bound, and then verify the assertions of \autoref{lem:validated} in order to establish the analogue of \autoref{thm:twolongestsides}.

\subsection{The algorithm}

\begin{figure}[t]
  \begin{center}
    
\hspace{\fill}
    \subfloat[Steps (1) and (2): $L^{(k)}$ shaded with red, $U^{(k)}$ magenta and $R_U$ green. Dirichlet boundary conditions on domain $D_U(k, p)$ as red lines, Neumann as blue lines. Test domain $D_L(k, p)$ is the complement of magenta and green.\label{fig:step12}]{
  \begin{tikzpicture}[scale=7]
      \draw (0,0) node [below] {\tiny $O$};
      \draw (0,0.7265625) node [above] {\tiny $B$};
      \draw (1,0) node [below] {\tiny $A$};
    \begin{scope}\draw[clip] (0,0.7265625)  |- (1,0)  --cycle;
\draw[help lines,yscale=0.7265625,opacity=0.5] (0,0) grid[step=0.015625] (1,1);\fill[red,fill opacity=0.2] (0.0, 0.0)--(0.375, 0.0)--(0.375, 0.158935546875)--(0.4375, 0.158935546875)--(0.4375, 0.2611083984375)--(0.484375, 0.2611083984375)--(0.484375, 0.3065185546875)--(0.578125, 0.3065185546875)--(1, 0)--(0, 0) -- cycle;
\fill[magenta,fill opacity=0.2] (0.0, 0.0)--(0.0, 0.1475830078125)--(0.28125, 0.1475830078125)--(0.28125, 0.1929931640625)--(0.3125, 0.1929931640625)--(0.3125, 0.2724609375)--(0.390625, 0.2724609375)--(0.390625, 0.340576171875)--(0.4375, 0.340576171875)--(0.4375, 0.40869140625)--(0, 0.7265625) -- cycle;

\fill[green!50!black,opacity=0.2] (0,0.7265625) rectangle (0.34375, 0.204345703125);
\draw (0.34375, 0.204345703125) node [shape=circle,fill=black,fill opacity=1,minimum size=1.5mm,inner sep=0pt,outer sep=0pt] {} node [below left=-3pt,black,opacity=1] {\tiny $p_U^{(k+1)}$};

\end{scope}\draw[very thick, red](0.0, 0.0)--(0.34375, 0.0)--(0.34375, 0.204345703125)--(0.4375, 0.204345703125)--(0.4375, 0.2611083984375)--(0.484375, 0.2611083984375)--(0.484375, 0.3065185546875)--(0.578125, 0.3065185546875)coordinate (a); \draw[very thick,blue] (a) -- (0,0.7265625) -- (0,0);
\end{tikzpicture}
}
\hspace{\fill}
    \subfloat[Steps (3) and (4): $U^{(k+1)}$ shaded with red, $L^{(k)}$ magenta and $R_L$ green. Dirichlet boundary conditions on domain $D_L(k+1, p)$ as red lines, Neumann as blue lines. Test domain $D_U(k,p)$ is the complement of magenta and green.\label{fig:step34}]{
\begin{tikzpicture}[scale=7]
      \draw (0,0) node [below] {\tiny $O$};
      \draw (0,0.7265625) node [above] {\tiny $B$};
      \draw (1,0) node [below] {\tiny $A$};
\begin{scope}\draw[clip] (0,0.7265625) |- (1,0) --cycle;

\draw[help lines,yscale=0.7265625,opacity=0.5] (0,0) grid[step=0.015625] (1,1);\fill[red,fill opacity=0.2] (0.0, 0.0)--(0.0, 0.1475830078125)--(0.28125, 0.1475830078125)--(0.28125, 0.1929931640625)--(0.3125, 0.1929931640625)--(0.3125, 0.204345703125)--(0.34375, 0.204345703125)--(0.34375, 0.2724609375)--(0.390625, 0.2724609375)--(0.390625, 0.340576171875)--(0.4375, 0.340576171875)--(0.4375, 0.40869140625)--(0, 0.7265625) -- cycle;
\fill[magenta,fill opacity=0.2] (0.0, 0.0)--(0.375, 0.0)--(0.375, 0.158935546875)--(0.4375, 0.158935546875)--(0.4375, 0.2611083984375)--(0.484375, 0.2611083984375)--(0.484375, 0.3065185546875)--(0.578125, 0.3065185546875)--(1, 0)--(0, 0) -- cycle;

\fill[green!50!black,opacity=0.2] (1,0) rectangle (0.40625, 0.2156982421875);
\draw (0.40625, 0.2156982421875) node [shape=circle,fill=black,fill opacity=1,minimum size=1.5mm,inner sep=0pt,outer sep=0pt] {} node [above right=-3pt,black,opacity=1] {\tiny $p_L^{(k+1)}$};

\end{scope}\draw[very thick, red](0.0, 0.1475830078125)--(0.28125, 0.1475830078125)--(0.28125, 0.1929931640625)--(0.3125, 0.1929931640625)--(0.3125, 0.204345703125)--(0.34375, 0.204345703125)--(0.34375, 0.2156982421875)--(0.40625, 0.2156982421875)--(0.40625, 0.340576171875)--(0.4375, 0.340576171875)--(0.4375, 0.40869140625)coordinate (a); \draw[very thick,blue] (a) -- (1,0) -- (0,0) -- (0.0, 0.1475830078125);
\end{tikzpicture}
}
\hspace{\fill}

  \end{center}
  \caption{Algorithm for constructing excluded domains.}
  \label{fig:alg}
\end{figure}
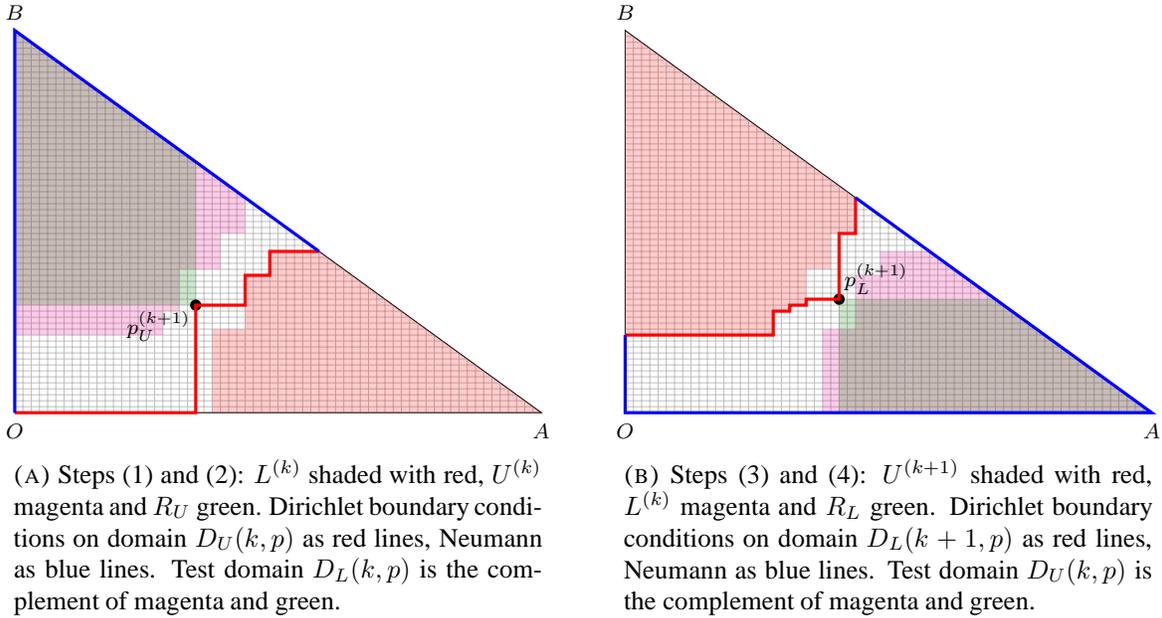
We now present a strategy, Algorithm \ref{alg}, for constructing certain subdomains $D_L, D_U, U^{(k)}, L^{(k)}$, $k=0,..$ of $T$. These are the subdomains for which we shall study the lowest mixed eigenvalues. The algorithm fails if both of the subdomains $L^{(k)}$ and $U^{(k)}$ cannot be constructed for some $k$. The steps of the algorithm are shown on \autoref{fig:alg}.

\begin{algorithm}[h]
\small  
\begin{algorithmic}[1]
\State Begin with $L^{(0)}$ and $U^{(0)}$ and set threshold value $V$ as in the previous section and let $k=0$.
\For{k=0,1,2..}
    \State Let $S=\emptyset$.
    \For{ all points $p$ on the grid $G$ {\it (See Remark \ref{Remark:linearsearch} below)}}
      \State Construct the domain $D_U(k, p)$. 
      \If {$\lambda_1(D_U(k, p))>V$}, add $p$ to the set $S$. \EndIf
     \EndFor
     \If{$S$ is empty}, $U^{(k+1)}:=U^{(k)},p_{U}^{k+1}=(0,0)$
     \Else
       \For{all points $p \in S$} construct the domain $D_L(k, p)$.
       \EndFor
    \State  Let $p_U^{(k+1)}$ be the point $p$ such that the smallest mixed eigenvalue of $D_L(k, p)$ is maximal (over $S$). See \autoref{fig:step12}. Point $p_U^{(k+1)}$ defines $U^{(k+1)}$.
   \EndIf

    \State Let $S=\emptyset$.
   \For{all points $p \in G$ {\it (See Remark \ref{Remark:linearsearch} below)}}
   \State Construct the domain $D_L(k+1, p)$ (we are using the newly created $U^{(k+1)}$).
   \If {$\lambda_1(D_L(k+1, p))>V$} add $p$ to the set $S$. \EndIf
   \EndFor
 
    \If{$S$ is empty and it was empty at step 7} \Return 'Algorithm failed' and exit \EndIf
    \If{$S$ is empty}, $L^{(k+1)}:=L^{(k)},p_{L}^{k+1}=(0,0)$
    \Else
	\For{all points $p \in S$} construct the domain $D_U(k, p)$. Here we are using $L^{(k)}$. 
	\EndFor
	\State Let $p_L^{(k+1)}$ be the point $p$ such that the smallest eigenvalue of $D_U(k, p)$ is maximal. See \autoref{fig:step34}. Point $p_L^{(k+1)}$ defines $L^{(k+1)}$. 
\EndIf
\If{The smallest mixed eigenvalue for $D_L(k, p_U^{(k+1)})$ or $D_U(k, p_L^{(k+1)})$ exceeds $V$} 
\State \Return 'Algorithm is successful' and terminate. 
\Else   $\;k \gets k+1$ and go back to step 2.\EndIf

\EndFor

\end{algorithmic}
\caption{Construction of of subdomains $U$ and $L$.}\label{alg}
\end{algorithm}

\clearpage

\begin{remarks}$\;$

\begin{enumerate}
\item  \label{Remark:linearsearch}   
	Note that it is not necessary to construct $D_U(k, p)$ for every grid point $p$. Indeed, if $p\in S$ then any point from $R_U(p)$ is also in $S$, by domain monotonicity \autoref{lem:harrell}. Similarly if $p\not\in S$, then neither are any points from $R_L(p)$. As a consequence, we only perform a linear search starting at the origin and moving up until we find the first grid point $p\in S$, then move one grid-point right and repeat.  
\item Steps 2 and 4 choose the ``best'' point $p$ to use to expand the excluded regions, with the aim of fine-tuning the algorithm to terminate in as few steps as possible. It is entirely possible that the length of the lists in \eqref{eq:pL} and \eqref{eq:pU} could be reduced by using different strategy.  On the other hand, the above formulation is simple, and more elaborate choices we tried did not lead to significant improvements.

\item We choose to grow the exclusion sets $U$ and $L$ by adding only one point $p$. We could simply add all points $p\in S$ to new exclusion sets. However, for every $p$ we use, we need to supply either an exact eigenvalue or a validated lower bound.  In other words the lists of points $p_L$ and $p_U$ could be replaced with lists of lists of points, and in every step the whole list of points would be used to grow $U^{(k)}$ and $L^{(k)}$.  This is certainly better from a numerical point of view, but this improvement would come at the expense of an extremely long proof certificate.

\end{enumerate}
\end{remarks}

\section{Validated lower bounds}\label{Section:Numerics}
There are two technical difficulties, which we resolve here, in obtaining rigorous numerical bounds for eigenvalues. By considering $T'\subset T\subset T''$ (see (\ref{eq:Tp}-\ref{eq:Tpp})) with rational coordinates in $T'$ and $T''$ we avoid floating point errors due to inexact domain representation. We further map  $T'' \rightarrow T'''$, a similar triangle with perpendicular sides 93 and 128. If $\lambda$ is an eigenvalue of $T'''$, then $128^2\lambda$ is an eigenvalue of $T''$.

Before we describe the validated numerical strategy, we list the sub-domains of $T'''$ for which we will compute the lowest mixed eigenvalues. Next, in \autoref{sec:NonconformingFiniteElements}, we fix notation and definitions for the numerical approximation strategy used.

\subsection{Polygonal domains for which we find validated lower bounds.}$\;$
\label{sec:polygonaldomain}
We now enumerate the upper and lower domains, by listing their vertices.  Note that for ease of reading these are listed on the $(64,64)$ right isosceles triangle $T_{ref}$. To obtain the coordinates of these points in $T'''$, scale the x-coordinate by 2 and the y-coordinate by $2\alpha=93/64$.

The upper domains $D_U(k,p_U^{(k+1)})$ (shown on \autoref{fig:upper}) are created using grid points $p_U$. We only list every other vertex on the step-like (red) portion of the boundary of the polygon and we skip the top vertex $(0, 64)$ (connecting two blue boundary pieces):
{\small
\begin{itemize}
\item[$k=0$]: (0, 0), (25, 24) 
\item[$k=1$]: (0, 0), (20, 17), (28, 21) 
\item[$k=2$]: (0, 0), (28, 30) 
\item[$k=3$]: (0, 0), (18, 13), (24, 14), (28, 21), (31, 27)
\item[$k=4$]: (0, 0), (22, 18), (28, 23), (31, 27), 
\item[$k=5$]: (0, 0), (24, 14), (26, 19), (28, 28), 
\item[$k=6$]: (0, 0), (21, 9),  (24, 14), (26, 19), (28, 23), (29, 29)
\item[$k=7$]: (0, 0), (19, 12), (24, 14), (26, 19), (28, 23), (30, 27)
\item[$k=8$]: (0, 0), (21, 9),  (23, 15), (26, 19), (27, 25), (30, 27)
\item[$k=9$]: (0, 0), (21, 9),  (23, 15), (25, 20), (28, 23), (30, 27), (31, 30)
\end{itemize}
}
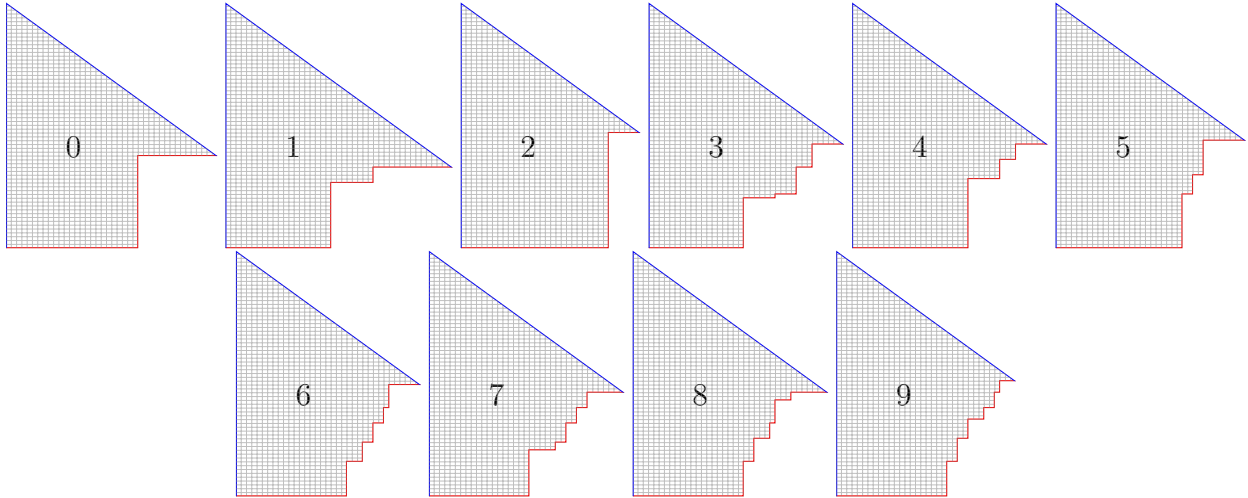
\begin{figure}[t]
  \begin{center}

\begin{tikzpicture}[scale=4.5]
\clip(0.0, 0.0)--(0.390625, 0.0)--(0.390625, 0.2724609375)--(0.625, 0.2724609375) -- (0,0.7265625);
\draw[thick, red](0.0, 0.0)--(0.390625, 0.0)--(0.390625, 0.2724609375)--(0.625, 0.2724609375)coordinate (a);
\draw[thick, blue] (a) -- (0,0.7265625) -- (0,0); 
\draw (0.2,0.3) node {$0$};
\draw[help lines,yscale=0.7265625,opacity=0.5] (0,0) grid[step=0.015625] (1,1);
\end{tikzpicture}
\begin{tikzpicture}[scale=4.5]
\clip(0.0, 0.0)--(0.3125, 0.0)--(0.3125, 0.1929931640625)--(0.4375, 0.1929931640625)--(0.4375, 0.2384033203125)--(0.671875, 0.2384033203125) -- (0,0.7265625);
\draw[thick, red](0.0, 0.0)--(0.3125, 0.0)--(0.3125, 0.1929931640625)--(0.4375, 0.1929931640625)--(0.4375, 0.2384033203125)--(0.671875, 0.2384033203125)coordinate (a);
\draw[thick, blue] (a) -- (0,0.7265625) -- (0,0);
\draw (0.2,0.3) node {$1$};
\draw[help lines,yscale=0.7265625,opacity=0.5] (0,0) grid[step=0.015625] (1,1);
\end{tikzpicture}
\begin{tikzpicture}[scale=4.5]
\clip(0.0, 0.0)--(0.4375, 0.0)--(0.4375, 0.340576171875)--(0.53125, 0.340576171875) -- (0,0.7265625);
\draw[thick, red](0.0, 0.0)--(0.4375, 0.0)--(0.4375, 0.340576171875)--(0.53125, 0.340576171875)coordinate (a);
\draw[thick, blue] (a) -- (0,0.7265625) -- (0,0);
\draw (0.2,0.3) node {$2$};
\draw[help lines,yscale=0.7265625,opacity=0.5] (0,0) grid[step=0.015625] (1,1);
\end{tikzpicture}
\begin{tikzpicture}[scale=4.5]
\clip(0.0, 0.0)--(0.28125, 0.0)--(0.28125, 0.1475830078125)--(0.375, 0.1475830078125)--(0.375, 0.158935546875)--(0.4375, 0.158935546875)--(0.4375, 0.2384033203125)--(0.484375, 0.2384033203125)--(0.484375, 0.3065185546875)--(0.578125, 0.3065185546875) -- (0,0.7265625);
\draw[thick, red](0.0, 0.0)--(0.28125, 0.0)--(0.28125, 0.1475830078125)--(0.375, 0.1475830078125)--(0.375, 0.158935546875)--(0.4375, 0.158935546875)--(0.4375, 0.2384033203125)--(0.484375, 0.2384033203125)--(0.484375, 0.3065185546875)--(0.578125, 0.3065185546875)coordinate (a);
\draw[thick, blue] (a) -- (0,0.7265625) -- (0,0);
\draw (0.2,0.3) node {$3$};
\draw[help lines,yscale=0.7265625,opacity=0.5] (0,0) grid[step=0.015625] (1,1);
\end{tikzpicture}
\begin{tikzpicture}[scale=4.5]
\clip(0.0, 0.0)--(0.34375, 0.0)--(0.34375, 0.204345703125)--(0.4375, 0.204345703125)--(0.4375, 0.2611083984375)--(0.484375, 0.2611083984375)--(0.484375, 0.3065185546875)--(0.578125, 0.3065185546875) -- (0,0.7265625);
\draw[thick, red](0.0, 0.0)--(0.34375, 0.0)--(0.34375, 0.204345703125)--(0.4375, 0.204345703125)--(0.4375, 0.2611083984375)--(0.484375, 0.2611083984375)--(0.484375, 0.3065185546875)--(0.578125, 0.3065185546875)coordinate (a);
\draw[thick, blue] (a) -- (0,0.7265625) -- (0,0);
\draw (0.2,0.3) node {$4$};
\draw[help lines,yscale=0.7265625,opacity=0.5] (0,0) grid[step=0.015625] (1,1);
\end{tikzpicture}
\begin{tikzpicture}[scale=4.5]
\clip(0.0, 0.0)--(0.375, 0.0)--(0.375, 0.158935546875)--(0.40625, 0.158935546875)--(0.40625, 0.2156982421875)--(0.4375, 0.2156982421875)--(0.4375, 0.31787109375)--(0.5625, 0.31787109375) -- (0,0.7265625);
\draw[thick, red](0.0, 0.0)--(0.375, 0.0)--(0.375, 0.158935546875)--(0.40625, 0.158935546875)--(0.40625, 0.2156982421875)--(0.4375, 0.2156982421875)--(0.4375, 0.31787109375)--(0.5625, 0.31787109375)coordinate (a);
\draw[thick, blue] (a) -- (0,0.7265625) -- (0,0);
\draw (0.2,0.3) node {$5$};
\draw[help lines,yscale=0.7265625,opacity=0.5] (0,0) grid[step=0.015625] (1,1);
\end{tikzpicture}
\begin{tikzpicture}[scale=4.5]
\clip(0.0, 0.0)--(0.328125, 0.0)--(0.328125, 0.1021728515625)--(0.375, 0.1021728515625)--(0.375, 0.158935546875)--(0.40625, 0.158935546875)--(0.40625, 0.2156982421875)--(0.4375, 0.2156982421875)--(0.4375, 0.2611083984375)--(0.453125, 0.2611083984375)--(0.453125, 0.3292236328125)--(0.546875, 0.3292236328125) -- (0,0.7265625);
\draw[thick, red](0.0, 0.0)--(0.328125, 0.0)--(0.328125, 0.1021728515625)--(0.375, 0.1021728515625)--(0.375, 0.158935546875)--(0.40625, 0.158935546875)--(0.40625, 0.2156982421875)--(0.4375, 0.2156982421875)--(0.4375, 0.2611083984375)--(0.453125, 0.2611083984375)--(0.453125, 0.3292236328125)--(0.546875, 0.3292236328125)coordinate (a);
\draw[thick, blue] (a) -- (0,0.7265625) -- (0,0);
\draw (0.2,0.3) node {$6$};
\draw[help lines,yscale=0.7265625,opacity=0.5] (0,0) grid[step=0.015625] (1,1);
\end{tikzpicture}
\begin{tikzpicture}[scale=4.5]
\clip(0.0, 0.0)--(0.296875, 0.0)--(0.296875, 0.13623046875)--(0.375, 0.13623046875)--(0.375, 0.158935546875)--(0.40625, 0.158935546875)--(0.40625, 0.2156982421875)--(0.4375, 0.2156982421875)--(0.4375, 0.2611083984375)--(0.46875, 0.2611083984375)--(0.46875, 0.3065185546875)--(0.578125, 0.3065185546875) -- (0,0.7265625);
\draw[thick, red](0.0, 0.0)--(0.296875, 0.0)--(0.296875, 0.13623046875)--(0.375, 0.13623046875)--(0.375, 0.158935546875)--(0.40625, 0.158935546875)--(0.40625, 0.2156982421875)--(0.4375, 0.2156982421875)--(0.4375, 0.2611083984375)--(0.46875, 0.2611083984375)--(0.46875, 0.3065185546875)--(0.578125, 0.3065185546875)coordinate (a);
\draw[thick, blue] (a) -- (0,0.7265625) -- (0,0);
\draw (0.2,0.3) node {$7$};
\draw[help lines,yscale=0.7265625,opacity=0.5] (0,0) grid[step=0.015625] (1,1);
\end{tikzpicture}
\begin{tikzpicture}[scale=4.5]
\clip(0.0, 0.0)--(0.328125, 0.0)--(0.328125, 0.1021728515625)--(0.359375, 0.1021728515625)--(0.359375, 0.1702880859375)--(0.40625, 0.1702880859375)--(0.40625, 0.2156982421875)--(0.421875, 0.2156982421875)--(0.421875, 0.2838134765625)--(0.46875, 0.2838134765625)--(0.46875, 0.3065185546875)--(0.578125, 0.3065185546875) -- (0,0.7265625);
\draw[thick, red](0.0, 0.0)--(0.328125, 0.0)--(0.328125, 0.1021728515625)--(0.359375, 0.1021728515625)--(0.359375, 0.1702880859375)--(0.40625, 0.1702880859375)--(0.40625, 0.2156982421875)--(0.421875, 0.2156982421875)--(0.421875, 0.2838134765625)--(0.46875, 0.2838134765625)--(0.46875, 0.3065185546875)--(0.578125, 0.3065185546875)coordinate (a);
\draw[thick, blue] (a) -- (0,0.7265625) -- (0,0);
\draw (0.2,0.3) node {$8$};
\draw[help lines,yscale=0.7265625,opacity=0.5] (0,0) grid[step=0.015625] (1,1);
\end{tikzpicture}
\begin{tikzpicture}[scale=4.5]
\clip(0.0, 0.0)--(0.328125, 0.0)--(0.328125, 0.1021728515625)--(0.359375, 0.1021728515625)--(0.359375, 0.1702880859375)--(0.390625, 0.1702880859375)--(0.390625, 0.22705078125)--(0.4375, 0.22705078125)--(0.4375, 0.2611083984375)--(0.46875, 0.2611083984375)--(0.46875, 0.3065185546875)--(0.484375, 0.3065185546875)--(0.484375, 0.340576171875)--(0.53125, 0.340576171875) -- (0,0.7265625);
\draw[thick, red](0.0, 0.0)--(0.328125, 0.0)--(0.328125, 0.1021728515625)--(0.359375, 0.1021728515625)--(0.359375, 0.1702880859375)--(0.390625, 0.1702880859375)--(0.390625, 0.22705078125)--(0.4375, 0.22705078125)--(0.4375, 0.2611083984375)--(0.46875, 0.2611083984375)--(0.46875, 0.3065185546875)--(0.484375, 0.3065185546875)--(0.484375, 0.340576171875)--(0.53125, 0.340576171875)coordinate (a);
\draw[thick, blue] (a) -- (0,0.7265625) -- (0,0);
\draw (0.2,0.3) node {$9$};
\draw[help lines,yscale=0.7265625,opacity=0.5] (0,0) grid[step=0.015625] (1,1);
\end{tikzpicture}

  \end{center}
  \caption{Upper domains from the algorithm.}
  \label{fig:upper}
\end{figure}
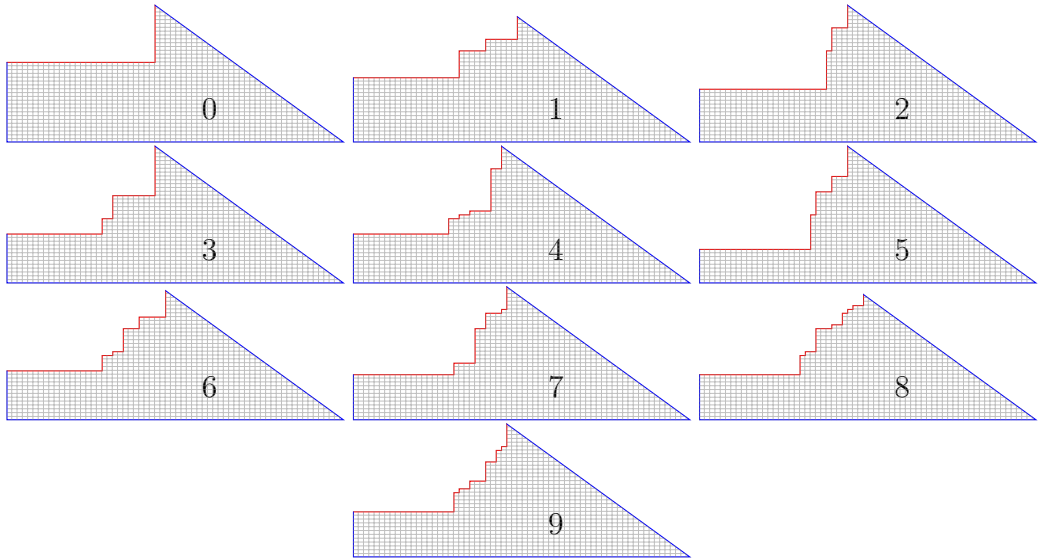
\begin{figure}[t]
  \begin{center}

\begin{tikzpicture}[scale=4.5]
\clip (1,0) -- (0,0) --(0.0, 0.2384033203125)--(0.4375, 0.2384033203125)--(0.4375, 0.40869140625);
\draw[thick, red](0.0, 0.2384033203125)--(0.4375, 0.2384033203125)--(0.4375, 0.40869140625)coordinate (a);
\draw[thick, blue] (a) -- (1,0) -- (0,0) -- (0.0, 0.2384033203125);
\draw (0.6,0.1) node {$0$};
\draw[help lines,yscale=0.7265625,opacity=0.5] (0,0) grid[step=0.015625] (1,1);
\end{tikzpicture}
\begin{tikzpicture}[scale=4.5]
\clip (1,0) -- (0,0) --(0.0, 0.1929931640625)--(0.3125, 0.1929931640625)--(0.3125, 0.2724609375)--(0.390625, 0.2724609375)--(0.390625, 0.3065185546875)--(0.484375, 0.3065185546875)--(0.484375, 0.3746337890625);
\draw[thick, red](0.0, 0.1929931640625)--(0.3125, 0.1929931640625)--(0.3125, 0.2724609375)--(0.390625, 0.2724609375)--(0.390625, 0.3065185546875)--(0.484375, 0.3065185546875)--(0.484375, 0.3746337890625)coordinate (a);
\draw[thick, blue] (a) -- (1,0) -- (0,0) -- (0.0, 0.1929931640625);
\draw (0.6,0.1) node {$1$};
\draw[help lines,yscale=0.7265625,opacity=0.5] (0,0) grid[step=0.015625] (1,1);
\end{tikzpicture}
\begin{tikzpicture}[scale=4.5]
\clip (1,0) -- (0,0) --(0.0, 0.158935546875)--(0.375, 0.158935546875)--(0.375, 0.2724609375)--(0.390625, 0.2724609375)--(0.390625, 0.340576171875)--(0.4375, 0.340576171875)--(0.4375, 0.40869140625);
\draw[thick, red](0.0, 0.158935546875)--(0.375, 0.158935546875)--(0.375, 0.2724609375)--(0.390625, 0.2724609375)--(0.390625, 0.340576171875)--(0.4375, 0.340576171875)--(0.4375, 0.40869140625)coordinate (a);
\draw[thick, blue] (a) -- (1,0) -- (0,0) -- (0.0, 0.158935546875);
\draw (0.6,0.1) node {$2$};
\draw[help lines,yscale=0.7265625,opacity=0.5] (0,0) grid[step=0.015625] (1,1);
\end{tikzpicture}
\begin{tikzpicture}[scale=4.5]
\clip (1,0) -- (0,0) --(0.0, 0.1475830078125)--(0.28125, 0.1475830078125)--(0.28125, 0.1929931640625)--(0.3125, 0.1929931640625)--(0.3125, 0.2611083984375)--(0.4375, 0.2611083984375)--(0.4375, 0.40869140625);
\draw[thick, red](0.0, 0.1475830078125)--(0.28125, 0.1475830078125)--(0.28125, 0.1929931640625)--(0.3125, 0.1929931640625)--(0.3125, 0.2611083984375)--(0.4375, 0.2611083984375)--(0.4375, 0.40869140625)coordinate (a);
\draw[thick, blue] (a) -- (1,0) -- (0,0) -- (0.0, 0.1475830078125);
\draw (0.6,0.1) node {$3$};
\draw[help lines,yscale=0.7265625,opacity=0.5] (0,0) grid[step=0.015625] (1,1);
\end{tikzpicture}
\begin{tikzpicture}[scale=4.5]
\clip (1,0) -- (0,0) --(0.0, 0.1475830078125)--(0.28125, 0.1475830078125)--(0.28125, 0.1929931640625)--(0.3125, 0.1929931640625)--(0.3125, 0.204345703125)--(0.34375, 0.204345703125)--(0.34375, 0.2156982421875)--(0.40625, 0.2156982421875)--(0.40625, 0.340576171875)--(0.4375, 0.340576171875)--(0.4375, 0.40869140625);
\draw[thick, red](0.0, 0.1475830078125)--(0.28125, 0.1475830078125)--(0.28125, 0.1929931640625)--(0.3125, 0.1929931640625)--(0.3125, 0.204345703125)--(0.34375, 0.204345703125)--(0.34375, 0.2156982421875)--(0.40625, 0.2156982421875)--(0.40625, 0.340576171875)--(0.4375, 0.340576171875)--(0.4375, 0.40869140625)coordinate (a);
\draw[thick, blue] (a) -- (1,0) -- (0,0) -- (0.0, 0.1475830078125);
\draw (0.6,0.1) node {$4$};
\draw[help lines,yscale=0.7265625,opacity=0.5] (0,0) grid[step=0.015625] (1,1);
\end{tikzpicture}
\begin{tikzpicture}[scale=4.5]
\clip (1,0) -- (0,0) --(0.0, 0.1021728515625)--(0.328125, 0.1021728515625)--(0.328125, 0.204345703125)--(0.34375, 0.204345703125)--(0.34375, 0.2724609375)--(0.390625, 0.2724609375)--(0.390625, 0.31787109375)--(0.4375, 0.31787109375)--(0.4375, 0.40869140625);
\draw[thick, red](0.0, 0.1021728515625)--(0.328125, 0.1021728515625)--(0.328125, 0.204345703125)--(0.34375, 0.204345703125)--(0.34375, 0.2724609375)--(0.390625, 0.2724609375)--(0.390625, 0.31787109375)--(0.4375, 0.31787109375)--(0.4375, 0.40869140625)coordinate (a);
\draw[thick, blue] (a) -- (1,0) -- (0,0) -- (0.0, 0.1021728515625);
\draw (0.6,0.1) node {$5$};
\draw[help lines,yscale=0.7265625,opacity=0.5] (0,0) grid[step=0.015625] (1,1);
\end{tikzpicture}
\begin{tikzpicture}[scale=4.5]
\clip (1,0) -- (0,0) --(0.0, 0.1475830078125)--(0.28125, 0.1475830078125)--(0.28125, 0.1929931640625)--(0.3125, 0.1929931640625)--(0.3125, 0.204345703125)--(0.34375, 0.204345703125)--(0.34375, 0.2724609375)--(0.390625, 0.2724609375)--(0.390625, 0.3065185546875)--(0.46875, 0.3065185546875)--(0.46875, 0.385986328125);
\draw[thick, red](0.0, 0.1475830078125)--(0.28125, 0.1475830078125)--(0.28125, 0.1929931640625)--(0.3125, 0.1929931640625)--(0.3125, 0.204345703125)--(0.34375, 0.204345703125)--(0.34375, 0.2724609375)--(0.390625, 0.2724609375)--(0.390625, 0.3065185546875)--(0.46875, 0.3065185546875)--(0.46875, 0.385986328125)coordinate (a);
\draw[thick, blue] (a) -- (1,0) -- (0,0) -- (0.0, 0.1475830078125);
\draw (0.6,0.1) node {$6$};
\draw[help lines,yscale=0.7265625,opacity=0.5] (0,0) grid[step=0.015625] (1,1);
\end{tikzpicture}
\begin{tikzpicture}[scale=4.5]
\clip (1,0) -- (0,0) --(0.0, 0.13623046875)--(0.296875, 0.13623046875)--(0.296875, 0.1702880859375)--(0.359375, 0.1702880859375)--(0.359375, 0.2724609375)--(0.390625, 0.2724609375)--(0.390625, 0.31787109375)--(0.4375, 0.31787109375)--(0.4375, 0.3292236328125)--(0.453125, 0.3292236328125)--(0.453125, 0.3973388671875);
\draw[thick, red](0.0, 0.13623046875)--(0.296875, 0.13623046875)--(0.296875, 0.1702880859375)--(0.359375, 0.1702880859375)--(0.359375, 0.2724609375)--(0.390625, 0.2724609375)--(0.390625, 0.31787109375)--(0.4375, 0.31787109375)--(0.4375, 0.3292236328125)--(0.453125, 0.3292236328125)--(0.453125, 0.3973388671875)coordinate (a);
\draw[thick, blue] (a) -- (1,0) -- (0,0) -- (0.0, 0.13623046875);
\draw (0.6,0.1) node {$7$};
\draw[help lines,yscale=0.7265625,opacity=0.5] (0,0) grid[step=0.015625] (1,1);
\end{tikzpicture}
\begin{tikzpicture}[scale=4.5]
\clip (1,0) -- (0,0) --(0.0, 0.13623046875)--(0.296875, 0.13623046875)--(0.296875, 0.1929931640625)--(0.3125, 0.1929931640625)--(0.3125, 0.204345703125)--(0.34375, 0.204345703125)--(0.34375, 0.2724609375)--(0.390625, 0.2724609375)--(0.390625, 0.2838134765625)--(0.421875, 0.2838134765625)--(0.421875, 0.31787109375)--(0.4375, 0.31787109375)--(0.4375, 0.3292236328125)--(0.453125, 0.3292236328125)--(0.453125, 0.340576171875)--(0.484375, 0.340576171875)--(0.484375, 0.3746337890625);
\draw[thick, red](0.0, 0.13623046875)--(0.296875, 0.13623046875)--(0.296875, 0.1929931640625)--(0.3125, 0.1929931640625)--(0.3125, 0.204345703125)--(0.34375, 0.204345703125)--(0.34375, 0.2724609375)--(0.390625, 0.2724609375)--(0.390625, 0.2838134765625)--(0.421875, 0.2838134765625)--(0.421875, 0.31787109375)--(0.4375, 0.31787109375)--(0.4375, 0.3292236328125)--(0.453125, 0.3292236328125)--(0.453125, 0.340576171875)--(0.484375, 0.340576171875)--(0.484375, 0.3746337890625)coordinate (a);
\draw[thick, blue] (a) -- (1,0) -- (0,0) -- (0.0, 0.13623046875);
\draw (0.6,0.1) node {$8$};
\draw[help lines,yscale=0.7265625,opacity=0.5] (0,0) grid[step=0.015625] (1,1);
\end{tikzpicture}
\begin{tikzpicture}[scale=4.5]
\clip (1,0) -- (0,0) --(0.0, 0.13623046875)--(0.296875, 0.13623046875)--(0.296875, 0.1929931640625)--(0.3125, 0.1929931640625)--(0.3125, 0.204345703125)--(0.34375, 0.204345703125)--(0.34375, 0.22705078125)--(0.390625, 0.22705078125)--(0.390625, 0.2838134765625)--(0.421875, 0.2838134765625)--(0.421875, 0.31787109375)--(0.4375, 0.31787109375)--(0.4375, 0.3292236328125)--(0.453125, 0.3292236328125)--(0.453125, 0.3973388671875);
\draw[red, thick](0.0, 0.13623046875)--(0.296875, 0.13623046875)--(0.296875, 0.1929931640625)--(0.3125, 0.1929931640625)--(0.3125, 0.204345703125)--(0.34375, 0.204345703125)--(0.34375, 0.22705078125)--(0.390625, 0.22705078125)--(0.390625, 0.2838134765625)--(0.421875, 0.2838134765625)--(0.421875, 0.31787109375)--(0.4375, 0.31787109375)--(0.4375, 0.3292236328125)--(0.453125, 0.3292236328125)--(0.453125, 0.3973388671875) coordinate (a);
\draw[thick, blue] (a) -- (1,0) -- (0,0) -- (0.0, 0.13623046875);
\draw (0.6,0.1) node {$9$};
\draw[help lines,yscale=0.7265625,opacity=0.5] (0,0) grid[step=0.015625] (1,1);
\end{tikzpicture}

  \end{center}
  \caption{Lower domains from the algorithm.}
  \label{fig:lower}
\end{figure}
Lower domains $D_L(k+1,p_U^{(k+1)})$ (shown on \autoref{fig:lower}) are created using points $p_L$. We again skip some vertices which can be deduces from the rest. 
{\small
\begin{itemize}
\item[$k=0$]:  (28, 21)
\item[$k=1$]: (20, 17), (25, 24), (31, 27)
\item[$k=2$]: (24, 14), (25, 24), (28, 30)
\item[$k=3$]: (18, 13), (20, 17), (28, 23)
\item[$k=4$]: (18, 13), (20, 17), (22, 18), (26, 19), (28, 30)
\item[$k=5$]: (21, 9),  (22, 18), (25, 24), (28, 28)
\item[$k=6$]: (18, 13), (20, 17), (22, 18), (25, 24), (30, 27)
\item[$k=7$]: (19, 12), (23, 15), (25, 24), (28, 28), (29, 29)
\item[$k=8$]: (19, 12), (20, 17), (22, 18), (25, 24), (27, 25), (28, 28), (29, 29), (31, 30)
\item[$k=9$]: (19, 12), (20, 17), (22, 18), (25, 20), (27, 25), (28, 28), (29, 29)
\end{itemize}
}

\subsection{Nonconforming finite elements method}
\label{sec:NonconformingFiniteElements}

Suppose $\Omega$ is any of the polygonal domains listed above in \autoref{sec:polygonaldomain}. Let $\Gamma_D$ be the Dirichlet part of its boundary. Since there is no closed-form expression for the eigenvalues, a lower bound for the smallest mixed eigenvalue $\lambda_1$ must be determined with the aid of approximation theory.

We can reformulate the mixed eigenvalue problem for $(u_j,\lambda_j)$ on the domain in variational form: find the $j$th eigenpair $(u_j,\lambda_j)$ so that for all test functions $v$ in the Sobolev space $H_{\Gamma_D}^1(\Omega)$, $$ \int_{\Omega}\nabla u_j \cdot \nabla v = \lambda_j \int_{\Omega} u_jv.$$
Choose a suitable finite-dimensional function space $V_h$ (where $h>0$ is called the {\it mesh parameter}), and consider  an approximation $u_{j,h} \in V_h$ of $u_j$, which satisfies the {\it discrete eigenvalue problem}: find the $j$th {\it discrete} eigenpair $(u_{j,h},\lambda_{j,h})$ so that for all test functions $v_h \in V_h $, $$ \int_{\Omega}\nabla u_{j,h} \cdot \nabla v_h = \lambda_{j,h} \int_{\Omega} u_{j,h}v_h.$$
By choosing $v_h$ to be basis functions of $V_h$, we obtain a discrete generalized eigenvalue system of the form
\begin{equation}\label{discretesystem}
\mathtt{K}_h\mathtt{u}_{j,h}=\lambda_{j,h}\mathtt{M}_h \mathtt{u}_{j,h}
\end{equation}The matrix $\mathtt{K_h}$ is called the {\it stiffness matrix} and the matrix $\mathtt{M}_h$ is the {\it mass matrix}.

Recall that if the integrals in \eqref{discretesystem} are computed exactly, and $V_h\subset H^1_{\Gamma_D}(\Omega)$, then the method is called {\it conforming}, otherwise it is {\it non-conforming}. 
 For a conforming method it is easy to see from the definition of the Rayleigh-Ritz quotient that $\lambda_{1}\leq \lambda_{1,h}$. However, we want a {\it lower} bound for $\lambda_1$. For this, we shall use a non-conforming method (following \cite{armentano,CG14}).
In what follows, we will be using a the well-known {\it Crouzeix-Raviart} linear non-conforming finite element discretization. 
Let $\hat{T}=QCD$ denote a right-angled triangle which is similar to $OAB''$, with a right angle at $Q$.  Let $|QC|=h, |QC|=\alpha h$ where $\alpha = \frac{93}{128}$. In what follows $h=2$. 
We tesselate a polygonal subdomain $\Omega\subset OAB''$ by a mesh $\Tau = \Tau(h)$ consisting of congruent copies $T$ of $\hat{T}$, ie, $OAB''=\bigcup_{T\in \Tau} T$. 
Recall that a linear function on $\hat{T}$ can be uniquely described by 
specifying its values at three points (these are its 'degrees of freedom'). For the Crouzeix-Raviart method we use, these degrees of freedom are on the (rational) midpoints of the edges of $\hat{T}$.
Let $\mathcal{E}$ denote the collection of all edges in the tesselation, and let $\Gamma_D$ denote the Dirichlet boundary of $\Omega$. 
 The finite-dimensional function space we use is
\begin{eqnarray*}V_n=CR_D^1(\Tau)&:=&\left\{v \in L^2(\Omega)\vert v\vert_{T}\in P_1(T), v \mbox{ is continuous at interior nodes} \right.\\
&& \left.\mbox{and 0 at the nodes on } \mathcal{E}(\Gamma_D)\right\}\end{eqnarray*}

It is obvious, now, that with this choice of $V_h$  the matrices in \eqref{discretesystem} are symmetric and that all entries are rational. Also, the mass matrix $\mathtt{M}_h$ is a diagonal matrix with diagonal blocks $\frac{h^2\alpha}{6} \mathtt{I}_{3\times3}.$ In our tesselation,  it is clear that each node $x_k$ is adjacent to at most 4 other nodes in $\Tau$ (recall that in the Crouzeix-Raviart discretization, each interior node lies at the midpoints of the sides of two triangles in the mesh, and its adjacent nodes are the midpoints of the other sides of those triangles). Therefore, $\mathtt{K}_h$ will have atmost 5 non-zero entries per row and column. By examining the {\it local stiffness matrix} for $\hat{T}\subset \Omega$, 
$$ [\hat{K}_h] = \left[\int_{\hat{T}} \nabla \phi_i\cdot \nabla \phi_j\right] = \left(\begin{matrix}
{2}/{\alpha} &-{2}/{\alpha}&0\\
-{2}/{\alpha}& 2(\alpha +{1}/{\alpha})&-2\alpha\\
0&-2\alpha & 2\alpha
\end{matrix}\right).$$
The entries of the full stiffness matrix $\mathtt{K}_h$ are built from these local contributions. The diagonal entries of $\mathtt{K}_h$ are, therefore, equal to $2\alpha,4\alpha, 2/\alpha, 4/\alpha, 2(\alpha + 1/\alpha), \text{ or }4(\alpha+1/\alpha)$. The off-diagonal terms are equal to $\pm 2\alpha$ or $\pm 2/\alpha$.
 The simple structure of $\mathtt{M}_h$ allows us to write
\begin{equation}\label{rationalmatrix}
 \mathtt{K}_h\mathtt{u}_{j,h} = \lambda_{j,h}\mathtt{M}_h \mathtt{u}_{j,h} = \frac{h^2\alpha}{6}\lambda_{j,h} \mathtt{B}_h \mathtt{u}_{j,h}\end{equation} where $\mathtt{B_h}$ is diagonal with 1 or 2 as the only nonzero entries. This matrix $\mathtt{B}_h$ is easy to factor and to invert. We also define a scaled version of the stiffness matrix which has integer entries $\mathtt{C_h} := (93)(64)\mathtt{K_h}$. We can then rewrite \eqref{rationalmatrix} as
 \begin{equation}\label{discretematrix}
 \mathtt C_h \mathtt u_{j,h} = (93)(64) \frac{h^2 \alpha}{6}\lambda_{j,h} \mathtt B_h\mathtt u_{j,h}= =\frac{93^2}{3}\lambda_{j,h} \mathtt{B}_{h}\mathtt{u}_{j,h}.
 \end{equation}

  We denote by $\mathtt{D}_h:=(\mathtt{B}_h)^{-1/2}$. This is a diagonal matrix with $1$ or $1/\sqrt{2}$ on the diagonals. With this notation we can rewrite the generalized eigenvalue problem in \eqref{rationalmatrix} to 
\begin{equation}\label{rescaleddiscrete}
    \mathtt{P}_h\mathtt{u}_{j,h}\equiv(\mathtt{D}_h\mathtt{K}_h \mathtt{D}_h) \mathtt{u}_{j,h}=\frac{h^2\alpha}{6} \lambda_{j,h} \mathtt{u}_{j,h} \equiv l_{j,h} \mathtt{u}_{j,h}
\end{equation}
The scaling in $h$ above is typically important from the point of view of performing stable computations; the conditioning of the original discrete system deteriorates with $h$.  However, note that we use the fixed value $h=2$.

Each of the stiffness matrices $\mathtt{K}_h$ and mass matrices $\mathtt{M}_h$ for $T'''$ and for the various subdomains $D_U(\cdot)$, $D_L(\cdot)$ has roughly $4000$ rows and columns, so we do not reproduce these matrices here.  Rather, we describe how the matrices were constructed. 
We first used {\tt FEniCS} \cite{LMW} to assemble the stiffness and mass matrices on $T'''$ (the full triangle, not the subdomains). We rescaled the resulting matrices to get integer entries, and transferred the results to a {\tt scipy.sparse} \cite{scipy} matrix for further processing, including handling of boundary conditions.  For the subdomains, to enforce the shape of a subdomain and the Dirichlet condition, we set the relevant degrees of freedom to $0$, forcing the solutions to equal $0$ on the complements of the subdomains from \autoref{lem:validated}. In practice, we accomplish this by removing the rows and columns for these degrees of freedom from both mass and stiffness matrices.

\subsection{Generalization of the validated nonconforming lower bounds.}

We now describe how to transform the above numerical scheme into rigorous lower bounds for mixed Dirichlet-Neumann eigenvalues. First, for the purpose of validated numerics, the entries of the matrices $\mathtt{K}_h, \mathtt{M}_h$ should be rational numbers. This can be done by ensuring the vertices of the domains above are rational, and by defining $V_h$ appropriately, which was done in the preceding section.  Next, the eigenvalues of large discrete systems such as \eqref{discretesystem} cannot be found in closed form. We must therefore ensure any approximation errors incurred are controlled in a manner that the desired lower bound is robust. We present three approaches in \autoref{sec:matrixlowerbound}~and~\autoref{sec:constructivelowerbound} below.

We begin by surveying the literature on such rigorous lower bounds.  In \cite{armentano}, an {\it asymptotic} result shows that provided $h$ is small enough, $\lambda_{1,h}$ computed using a (specific) non-conforming method provides a lower bound for the first eigenvalue $\lambda_1$ for the Dirichlet problem. It is only recently that a {\it computable}  lower bound based on a non-conforming finite element method for the first Dirichlet eigenvalue of the Laplacian on a polygonal domain has become available, in \cite{CG14}. This is achieved by approximating the first Dirichlet eigenfunctions using piecewise linear Crouzeix-Raviart elements, that is, approximating the first eigenfunction by functions from 
\begin{eqnarray*}CR_0^1(\Tau)&:=\left\{v \in L^2(\Omega)\vert v\vert_{T}\in P_1(T), v \mbox{ continuous at the midpoints of interior edges} \right.\\
& \left.\mbox{and 0 at the midpoints of boundary edges}\right\}\end{eqnarray*}
Suppose $CR_0^1(\Tau)$ is used to construct the generalized eigenvalue problem of form \eqref{discretesystem} for the pure Dirichlet problem. Barring very special circumstances, the eigenvalues of this system cannot be found in closed form. 
Therefore, suppose we use an iterative approximation technique to obtain an {\it approximation} $\tilde{\lambda}_{CR,1}$ to $\lambda_{1,h}$ (which is, itself, an approximation to $\lambda_1$). 
The authors of \cite{CG14} establish the computable inequality Theorem 3.1 of their paper: if $(\tilde{\lambda}_{CR,1}, \tilde{u}_{CR,1})\in \mathbb{R}\times CR_0^1(\Tau)$ approximate the first eigenpair $(\lambda_1, u_1)$ (with eigenfunction normalized to 1), 
if $\mathbf{r}$ is the algebraic residual, and if $\tilde{\lambda}_{CR,1}$ is closer to the first true discrete eigenvalue $\lambda_{CR,1} $ than to any other discrete eigenvalue (see \cite[Lemma 3.8]{CG14}), then 
    
    \begin{equation}\label{eqn:Carstensen}
	\frac{\tilde{\lambda}_{CR,1}-|\mathbf{r}|_{B^{-1}}}{1+\kappa^2(\tilde{\lambda}_{CR,1}-|\mathbf{r}|_{B^{-1}})H^2}\leq \lambda_1.
\end{equation} 

Here, $H$ is the maximal diameter in the regular triangulation $\Tau$ and $\kappa$ is a universal constant given by $\kappa^2 = \frac{1}{8} + j_{1,1}^{-2} \leq 0.1931$ for the first positive root $j_{i,1}$ of the Bessel function of the first kind.

In our problem, we need to establish a similar lower bound for the first eigenvalue of the Laplacian with {\it mixed} data. We need to therefore extend Theorem 3.1 from \cite{CG14} to this case. Specifically, we seek approximations $u_{CR,1}\in CR_D^1(\mathcal{T})$ for the first eigenfunction in a  mixed eigenvalue problem.

The non-conforming interpolant used to establish the interpolation estimate of Theorem 2.1 in \cite{CG14} remains unaltered except for the obvious change in domain and range: $\mathcal{I}_{NC}:H^1_{\Gamma_D}\rightarrow CR^1(\Tau)$ is defined on all edges $E$ which are not on $\Gamma_D$ as:
$$ I_{NC}v(mid(E)):=\frac{1}{|E|}\int_E v\, ds.$$ Theorem 2.1 also holds, with $\kappa$ unaltered. Recall that this $\kappa$ is the optimal Poincar\'e constant on a triangle \cite{LSminN}. Provided $\Gamma_D$ is non-empty, the $H^1(D)$ semi-norm remains a full norm (a fact used in the argument of \cite{CG14}). Therefore, the same inequality \eqref{eqn:Carstensen} holds.

Note, the lower bounds by Carstensen and Gedicke are conditional on algebraic eigenvalue approximation being close the smallest true eigenvalue (see \cite[Lemma 3.8]{CG14}), and this condition is not easy to validate. Nevertheless, their Theorem 3.2, relating exact discrete eigenvalue $\lambda_{CR,1}$ and $\lambda_1$, generalizes unconditionally to mixed eigenvalues:
\begin{lemma} Let $(\lambda_{CR,1})$ be the true eigenvalue corresponding to a shape-regular Crouzeix-Raviart discretization of the mixed eigenvalue problem for the Laplacian on a polygonal domain, with maximal mesh size $H$. Let $\lambda_1$ be the first eigenvalue of the mixed eigenvalue problem on this domain. Then 
\begin{align}
\label{eqn:nilima-lower-bound}
    \frac{\lambda_{CR,1}}{1+\kappa^2\lambda_{CR,1}H^2}\le \lambda_1.
\end{align}
Moreover, if $\tilde{\lambda}_{CR,1}$ is an approximation to the discrete eigenvalue $\lambda_{CR,1}$ which is closer to $\lambda_{CR,1}$ than to $\lambda_{CR,2}$, then again  \begin{equation}
	\frac{\tilde{\lambda}_{CR,1}-|\mathbf{r}|_{B^{-1}}}{1+\kappa^2(\tilde{\lambda}_{CR,1}-|\mathbf{r}|_{B^{-1}})H^2}\leq \lambda_1.
\end{equation}

\end{lemma}
Furthermore, the left hand side of~\eqref{eqn:nilima-lower-bound} is increasing in $\lambda$ allowing us to find a lower bound for $\lambda_1$ using a lower bound for the discrete eigenvalue $\lambda_{CR,1}$. In particular 
\begin{align}\label{lowereigbound}
    \Lambda< \lambda_{CR,1}\quad\Longrightarrow\quad\frac{\Lambda}{1+\kappa^2\Lambda H^2}< \lambda_1.
\end{align}
Note that we have $H^2=(h/64)^2+(\alpha h/64)^2$ and $\kappa^2\leq0.1931$ (recalling that we take $h = 2$ and $\alpha = 93/128)$. 
To obtain a proof of \autoref{lem:validated}, we need $V=12.25<\lambda_1$. Therefore it is enough to show that $\lambda_{CR,1}\ge 12.25+3/256=:\Lambda$. 
This choice of $\Lambda$ is made such that if we multiply both sides of the discrete eigenvalue problem by a power of $2$, then all matrices involved in (26) have integer entries. No floating point errors are incurred at the stage of matrix assembly.

As is standard for discrete eigenvalue computations, the quality of approximation to the first discrete eigenvalue $\tilde{\lambda}_{CR,1}$ depends (inversely) on the spectral gap between the first  and second discrete eigenvalues, $\lambda_{CR,2}-\lambda_{CR,1}$.

\subsection{Lower bound for the smallest eigenvalue of a sparse matrix}
\label{sec:matrixlowerbound}
Here we present a method of finding a verifiable lower bound for the lowest eigenvalue of a sparse rational matrix, without computing any eigenvalues. We use this method on the stiffness and mass matrix systems described above, and the domains from \autoref{lem:validated}. 

We have several generalized eigenvalue problems $Kx=\lambda_{CR,1}Mx$, one for each domain from \autoref{lem:validated}. We want to show that the smallest eigenvalue of each exceeds $\Lambda$. Instead of finding an approximate lowest eigenvalue (and checking that it is actually the smallest), we may prove that $K - \Lambda M$ is positive definite.
$K-\Lambda M$ is positive definite if and only if $\sqrt{M^{-1}}K\sqrt{M^{-1}}-\Lambda I$ is positive definite. If the latter is positive definite than the smallest eigenvalue of $\sqrt{M^{-1}}K\sqrt{M^{-1}}$ is larger than $\Lambda$, and it is the same as the smallest eigenvalue for the generalized system $Kx-\gamma Mx$.

\begin{proposition}
The matrices $K-\Lambda M$ are all positive definite.
\end{proposition}

\begin{proof}
Since $\Lambda$ has been chosen to be a rational number, this is a routine check with standard software.  For the diligent reader who wishes to check this on their own, there are several techniques, all of which we have tried successfully: one can use the $LU$ and $LDL'$ decomposition in interval arithmetic, or one can compute an exact LU decomposition with rational arithmetic.  Here are the details of how to check the proposition with standard software.

The matrix $K-\Lambda M$ has rational entries and is sparse. We used {\tt CHOLMOD} \cite{cholmod} to find a fill-reducing permutation (a matrix $P$ which makes the $LDL'$ decomposition of $P(K-\gamma M)P'$ as sparse as possible). Then we find the $LDL'$ or $LU$ decomposition of the permuted matrix. A matrix is positive definite if all diagonal entries in $D$ (or, respectively, in $U$) are positive.

To do the computation with interval arithmetic, we used the {\tt mpmath.mpi} \cite{mpmath} interval arithmetic library and the modified {\tt sympy} \cite{sympy} function {\tt SparseMatrix.LDLdecomposition} to check positivity of the entries in $D$.  To do the LU decomposition with exact rational arithmetic, we used sparse matrices from {\tt sage} \cite{sage} and its built-in LU method.  The second approach gives exact values for the diagonal entries of $U$, and again, all entries are positive.
\end{proof}


\subsection{Strategy for computing discrete eigenvalues}\label{validatedcomputing}
\label{sec:constructivelowerbound}

The nonconstructive method described the previous section already provides two ways to prove \autoref{lem:validated}. Nevertheless, in this section, we also present a constructive approach.  As discussed above, the eigenvalues of the discrete system \eqref{rescaleddiscrete} are not available in closed form. An iterative method will be required to compute an approximation to the eigenvalues.  We then convert the resulting approximation to bounds for the eigenvalues using interval arithmetic.

\begin{proposition}
    The smallest two mixed eigenvalues for the domains $D_L(\cdot)$, $D_U(\cdot)$ lie within the ranges shown in \autoref{table:upper domains 1st and 2nd},~\ref{table:lower domains 1st and 2nd}.
\end{proposition}

Using interval arithmetic, we first convert $\mathtt{P}_h$ to tridiagonal form (with intervals for each of the non-zero matrix entries), and then use a bisection method to obtain an approximation to $\frac{h^2\alpha}{6}\lambda_{1,h}$. This method uses the Sturm sequence property~\cite{TB97}, and allows us to locate other eigenvalues as well. In particular, we also obtain validated approximations to $\lambda_{2,h} $.

Then, using interval arithmetic, the Lanczos method~\cite{La50,GvL} on $\mathtt{P}_h $, and a validated nonlinear solver we obtain an approximation with error bound for $\lambda_{1,h}$.

For the latter two calculations, we use an interval arithmetic calculation in {\tt INTLAB} \cite{Ru99a}. Specifically, we begin with the rational matrices $\mathtt{K}_h,\mathtt{M}_h$; all subsequent matrix operations are done within interval arithmetic.

The use of the bisection method and the Sturm sequence property ensures that we have {\it no} missing eigenvalues in the initial part of the spectrum of \eqref{rescaleddiscrete}, and that the first two discrete eigenvalues are seperated: $0<l_{1,h}< l_{2,h}\leq l_{3,h}$. As it standard with this method, the convergence rate is slow. In each of the subdomains under consideration, we used this method to verify that (i) the first eigenvalue is simple and (ii) the (normalized) spectral gap $\lambda_{2,h}-\lambda_{1,h}$ is large. This method also provides an initial estimate for $\lambda_{1,h}$. 

With this initial estimate for $\lambda_{1,h}$, we use the Lanczos iteration on the generalized eigenvalue problem \eqref{rationalmatrix}, 
\[\mathtt{K}_h \mathtt{u}_{1,h} = \lambda_{1,h} \mathtt{M}_h \mathtt{u}_{1,h},\]  
to compute a higher-accuracy approximation to $\lambda_{1,h}$ and a good approximation to the corresponding eigenvector. Thanks to the (large) spectral gap between the first and the second discrete eigenvalue, we know the first eigenvector will be well-approximated. The specific implementation we use is the {\tt eigs} subroutine in {\tt Matlab}, which gives as output the eigenpair $(\tilde{\lambda}_{1,h}, \tilde{\mathtt{u}}_{1,h})$.  The validated numerics {\tt INTLAB} algorithm {\tt VerifyEig} is then applied to the eigenpair $(\frac{6}{h^2 \alpha}l_{1,h}, \bar{\mathtt{u}}_{1,h})$. The output, using validated interval numerics, is the approximate (interval) eigenpair $\bar{\lambda_{1,h}},\bar{u_{1,h}}$. 
The assertion follows from the Banach contraction mapping principle, \cite{rump2}. Therefore, the true eigenvalue $\lambda_{1,h}\subset\{ \bar{\lambda}_{1,h} \pm \epsilon\}$ where $\epsilon$ is the radius of the eigenvalue inclusion interval. 

In \autoref{table:upper domains 1st and 2nd} and \autoref{table:lower domains 1st and 2nd},
we record the midpoints $\bar{\lambda}_{1,h}, \bar{\lambda}_{2,h}$ of the first and second eigenvalues for each of the domains. We report the {\it quality of the first eigenvector}:  $\delta$ is the largest (component-wise) relative error (radius of interval scaled by midpoint). 
In \autoref{table:upper domains 1st} and \autoref{table:lower domains 1st}
we record first eigenvalue (midpoint of interval) $\bar{\lambda}_{1,h}$, the radius $\epsilon$ of this interval, and the quantity $\bar{\lambda}_{1,h}-\lambda-L*$. It is clear that 
$$ \bar{\lambda}_{1,h}-\lambda-L*> \epsilon>0.$$ Since the true discrete eigenvalue is provably guaranteed to lie $ \in \{ \bar{\lambda}_{1,h} \pm \epsilon\}$, we conclude that $\lambda_{1,h}> L*$. Consequently, $\Lambda>\gamma$.

\begin{table}
    \begin{tabular}{@{}cccc@{}}
	\toprule
Domain&  $\bar{\Lambda}_{1,h} $&$\bar{\Lambda}_{2,h} $ &$128^2\delta$\\
	\midrule
0 & 12.32808937 & 61.28234427 & 3.73706385e-08 \\  
1 & 12.28049706 & 63.11306391 & 3.62476864e-08 \\  
2 & 12.27067982 & 59.72331787 & 3.86656900e-08 \\  
3 & 12.28475539 & 63.34657906 & 3.61297611e-08 \\  
4 & 12.32978563 & 62.41415118 & 3.66665220e-08 \\  
5 & 12.27032723 & 60.89756087 & 3.73813582e-08 \\  
6 & 12.26881872 & 61.32799996 & 3.71251840e-08 \\  
7 & 12.26612593 & 62.64435623 & 3.63766466e-08 \\  
8 & 12.28907295 & 61.74502565 & 3.69613229e-08 \\  
9 & 12.28227434 & 61.77654477 & 3.70115832e-08 \\ 
	\bottomrule
    \end{tabular}\caption{$\bar{\Lambda}_{1,h}=128^2\bar{\lambda}_{1,h},\bar{\Lambda}_{2,h}=128^2 \bar{\lambda}_{2,h},\delta$ for the upper domains
\label{table:upper domains 1st and 2nd}
}  
\end{table}

\begin{table}
    \begin{tabular}{@{}cccc@{}}
    \toprule
Domain&  $\bar{\Lambda}_{1,h} $& $\bar{\Lambda}_{2,h} $ &$128^2\delta$\\ 
    \midrule
0 & 12.35848215 & 47.63909357 & 6.49892877e-08 \\  
1 & 12.27959016 & 56.23540471 & 5.64489920e-08 \\  
2 & 12.41003685 & 63.87608517 & 4.84346952e-08 \\  
3 & 12.27157804 & 61.18083146 & 4.90587904e-08 \\  
4 & 12.32234011 & 62.36099167 & 4.82268720e-08 \\  
5 & 12.30303087 & 64.35809373 & 4.12266778e-08 \\  
6 & 12.29807616 & 61.14841468 & 4.91463301e-08 \\  
7 & 12.30521249 & 63.27332626 & 4.61834285e-08 \\  
8 & 12.29850781 & 62.20762852 & 4.69517628e-08 \\  
9 & 12.29425300 & 62.65668399 & 4.66658156e-08 \\  
    \bottomrule
\end{tabular}\caption{$\bar{\Lambda}_{1,h}=128^2\bar{\lambda}_{1,h},\bar{\Lambda}_{2,h}=128^2 \bar{\lambda}_{2,h},\epsilon$ for the lower domains
\label{table:lower domains 1st and 2nd}
}  
\end{table}

\begin{table}
    \begin{tabular}{@{}cccc@{}}
    \toprule
Domain&  $\bar{\Lambda}_{1,h}$& $128^2\epsilon$&$\bar{\Lambda}_{1,h}-128^2\epsilon-(12.25+3/256)$ \\ 
    \midrule
0& 12.32808937 & 4.14743795e-11 & 6.63706231e-02 \\  
1 & 12.28049706 & 4.44586590e-11 & 1.87783061e-02 \\  
2 & 12.27067982 & 4.25988134e-11 & 8.96106588e-03 \\  
3& 12.28475539 & 4.13429291e-11 & 2.30366406e-02 \\  
4& 12.32978563 & 4.97983876e-11 & 6.80668799e-02 \\  
5& 12.27032723 & 5.32356381e-11 & 8.60847819e-03 \\  
6& 12.26881872 & 3.81223941e-11 & 7.09997069e-03 \\  
7& 12.26612593 & 3.64668296e-11 & 4.40717954e-03 \\  
8& 12.28907295 & 3.70494746e-11 & 2.73542003e-02 \\  
9& 12.28227434 & 4.88817875e-11 & 2.05555934e-02 \\  
    \bottomrule
\end{tabular}
\caption{$\Lambda_{1,h}, 128^2\epsilon, \Lambda_{1,h}128^2\epsilon -(12.25+3/256)$, for the upper domains. 
\label{table:upper domains 1st}
}
\end{table}
\begin{table}
    \begin{tabular}{@{}cccc@{}}
	\toprule
Domain&  $\bar{\Lambda}_{1,h}$& $128^2\epsilon$&$\bar{\Lambda}_{1,h}-128^2\epsilon-(12.25+3/256)$ \\ 
	\midrule
0 &12.35848215 & 4.66915395e-11 & 9.67633998e-02\\  
1& 12.27959016 & 4.41691128e-11 & 1.78714114e-02\\  
2 &12.41003685 & 4.30731006e-11 & 1.48318103e-01\\  
3& 12.27157804 & 4.69331241e-11 & 9.85928892e-03\\  
4 &12.32234011 & 4.10835810e-11 & 6.06213576e-02\\  
5& 12.30303087 & 4.15116830e-11 & 4.13121178e-02\\  
6& 12.29807616 & 5.08979525e-11 & 3.63574115e-02\\  
7& 12.30521249 & 4.22684110e-11 & 4.34937388e-02\\  
8& 12.29850781 & 4.78745932e-11 & 3.67890648e-02\\  
9& 12.29425300 & 4.56950033e-11 & 3.25342545e-02\\  
	\bottomrule
\end{tabular}
\caption{$\Lambda_{1,h}, 128^2\epsilon, \Lambda_{1,h}128^2\epsilon -(12.25+3/256)$, for the lower domains. 
\label{table:lower domains 1st}
}
\end{table}

\end{document}